\documentclass[12pt]{amsart}
\usepackage{xspace,amssymb,amsfonts,amsthm,ifpdf,fge}
\usepackage{epsfig,euscript,enumerate,epstopdf,microtype}
\usepackage[vcentermath]{youngtab}
\usepackage[all]{xy}
\usepackage[margin=3cm,footskip=25pt,headheight=20pt]{geometry}
\usepackage{setspace,ifpdf,tikz}
 \author{Catharina Stroppel}
 \address{Mathematisches Institut\\
          Beringstrasse 1\\
          Universit\"at Bonn\\
          53115 Bonn, Germany}
 \email{stroppel@math.uni-bonn.de}
 \urladdr{http://www.math.uni-bonn.de/people/stroppel/}
 \thanks{The first author is supported by the NSF and the Minerva Research Foundation grant DMS-0635607}
 \author{Ben Webster}
 \title[2-block Springer fibers]{2-block Springer fibers:\\ convolution algebras and coherent sheaves}
 \address{Department of Mathematics, Massachusetts Institute of Technology} 
\email{bwebster@math.mit.edu}
\urladdr{http://math.mit.edu/~bwebster}
 \thanks{The second author is supported by a National Science Foundation Postdoctoral Research
   Fellowship.}

\numberwithin{equation}{section}
\newcommand{\up}{\wedge}
\newcommand{\down}{\vee}
\newcommand{\nc}{\newcommand}
\nc{\renc}{\renewcommand}
\nc{\Cob}{\mathfrak{C}}
\nc{\Coh}{\mathsf{Coh}}
\nc{\Vect}{\mathsf{Vect}}
\nc{\hd}{\Omega^{1/2}}
\nc{\canb}{\Omega}
\nc{\op}{\operatorname}
\nc{\dn}{R}
\nc{\Fl}{F}
\nc{\ST}{\mathcal{K}}
\nc{\mg}{\mathfrak{g}}
\nc{\mZ}{\mathbb{Z}}
\nc{\mE}{\mathcal{E}}
\nc{\mb}{\mathfrak{b}}
\nc{\p}{\mathfrak{p}}
\nc{\mm}{\mathfrak{m}}
\nc{\mn}{\mathfrak{n}}
\nc{\ml}{\mathfrak{l}}
\nc{\Ex}{\mathfrak{Ex}}
\nc{\cU}{\mathcal{U}}
\nc{\cC}{\mathcal{C}}
\nc{\cM}{\mathcal{S}}
\nc{\mh}{\mathfrak{h}}
\nc{\Sym}{\mathrm{Sym}}
\nc{\cO}{\mathcal{O}}
\nc{\cI}{\mathcal{I}}
\nc{\can}{\mathrm{can}}
\nc{\Z}{\mathfrak{Z}}
\nc{\onto}{twoheadrightarrow}
\nc{\cE}{\mathcal E_c}
\nc{\en}{N} \nc{\flag}{X}
\nc{\rk}{\mathrm{rk}}
\nc{\spf}{Y}
\nc{\CP}{\mathbb{P}}
\nc{\com}[1]{Y_{#1}}
\nc{\la}{\lambda}
\nc{\mC}{\mathbb{C}}
\nc{\Sm}[1]{{#1}_\vee}
\nc{\Sp}[1]{{#1}_\wedge}
\nc{\si}{\sigma}
\nc{\de}{\delta}
\nc{\bigwedgie}[1]{\sideset{}{^{#1}}\bigwedge}
\nc{\im}{\mathrm{im}}
\nc{\C}{\mathbb{C}}
\nc{\Bx}{\mathbf{x}}
\nc{\m}{\mathbf{m}}
\nc{\too}{\longrightarrow}
\nc{\ep}{\epsilon}
\nc{\tY}{\tilde Y}
\nc{\taY}{\tilde{\EuScript{Y}}}
\nc{\taYt}{\taY^{(3)}}
\nc{\taYf}{\taY^{(4)}}
\nc{\al}{\alpha}
\nc{\be}{\beta} \nc{\KH}{\mathcal{H}} \nc{\vp}{\varphi} \nc{\LB}{V}
\nc{\End}{\operatorname{End}}
\nc{\Ext}{\operatorname{Ext}}
\nc{\sExt}{\mathcal{E}xt}
\nc{\Hom}{\operatorname{Hom}}
\nc{\fx}[1]{\mathcal{F}_\bullet({#1})}
\nc{\fxi}[2]{\mathcal{F}_{#1}(#2)} \nc{\st}[1]{\EuScript{Y}_{#1}}
\nc{\scr}{\mathcal} \nc{\re}{\sim} \nc{\sip}{\epsilon} \nc{\St}[1]{S(#1)}
\nc{\nxt}{\cup\cup}
\nc{\nested}{\Cup}
\nc{\tP}{\EuScript{P}}
\nc{\tQ}{\EuScript{Q}}
\nc{\Snk}{\EuScript{S}_{n-k,k}}
\theoremstyle{margin}
  \newtheorem{theorem}{Theorem}
 \newtheorem*{theorem*}{Theorem}
  \newtheorem{definition}[theorem]{Definition}
  \newtheorem{lemma}[theorem]{Lemma}
  
  \newtheorem{proposition}[theorem]{Proposition}
 \newtheorem*{definition*}{Definition}
  \newtheorem{corollary}[theorem]{Corollary}
  \newtheorem{remark}[theorem]{Remark}
  \newtheorem{question}[theorem]{Question}
  \newtheorem{ex}[theorem]{Example}
  \newtheorem{conjecture}[theorem]{Conjecture}
 \newtheorem*{conjecture*}{Conjecture}
\begin{document}
\begin{abstract}
For a fixed 2-block Springer fiber, we describe the
  structure of its irreducible components and their relation to the
  Bia\l ynicki-Birula paving, following work of Fung.
That is, we consider the space of complete flags in $\C^n$ preserved by a fixed nilpotent matrix with 2 Jordan blocks, and study the action of diagonal
matrices commuting with our fixed nilpotent. In particular, we describe the structure of each component, its set of torus fixed points, and prove a conjecture of Fung describing the intersection of any pair.

Then we define a convolution algebra structure on the direct sum of the
  cohomologies of pairwise intersections of irreducible components and
  closures of $\mC^*$-attracting sets (that is Bia\l ynicki-Birula cells), and show this is isomorphic to
 a generalization of the arc algebra of Khovanov defined by the first author.
 We investigate the connection
  of this algebra to Cautis \& Kamnitzer's recent work on link
  homology via coherent sheaves and suggest directions for future
  research.
\end{abstract}
\maketitle
\setcounter{tocdepth}{1}
\tableofcontents

\section*{Introduction}

Many important algebras arising in representation theory (Hecke
algebras, universal enveloping algebras, etc.) have a geometric
description based on convolution products.

Besides their intrinsic interest, realizing an algebra in terms of
convolution allows for a geometric understanding of the representation
theory of that algebra, in particular, the construction of collections
of standard and costandard modules, indicating the existence of an
interesting representation theory along the lines of highest weight
categories or quasi-hered\-itary algebras.  This approach has been
applied with great success to the representation theory of Weyl
groups, Hecke algebras of various flavors and universal enveloping
algebras, as is ably documented in the book of Chriss and Ginzburg
\cite{CG97}.

Using 2-block Springer fibers we present a construction of a family of
convolution algebras with a somewhat different nature than the above
examples (see Section~\ref{sec:convolution-algebras} for a precise
description).  With a certain specific choice of parameters and the
two Jordan blocks of the same size, the algebra is related to the
Ext-algebra of certain {\it coherent sheaves} on a resolution of the
corresponding Slo\-do\-wy slice and to a graphically defined algebra,
called the {\it arc algebra} $\KH^\bullet$, introduced by Khovanov
\cite{KhoJones}. For the general 2-Jordan-block case (not necessarily
equally sized), we establish an
isomorphism to the more general version of the arc algebra as
introduced in \cite{StrSpringer} and \cite{ChK}.

Our construction is built on a careful explicit geometric and
combinatorial analysis of the geometry of the Springer fiber and its
components. Apart from the 2-Jordan-block case, the structure of
irreducible components of Springer fibers is not sufficiently well
understood to generalize this construction, though significant
progress on the structure of components and their intersections has
been achieved in the square-zero (i.e. two column) case studied in
\cite{Melnikov}, in addition to the 2-Jordan-block case studied here.\\

Khovanov used his arc algebra to define a categorification of the
Jones polynomial (\cite{KhoJones}), followed by a representation theoretic
categorification of the Jones polynomial and the Reshetikhin-Turaev
$U(\mathfrak{sl}_2)$-tangle invariant obtained by the first author in
\cite{StDuke}. The choice of Jordan block sizes
corresponds there to a choice of a specific weight space in a tensor
product of many copies of the natural $U(\mathfrak{sl}_2)$-module,
hence naturally extends the case of two blocks of the same size. It is
known that after restriction to a suitable subcategory, this
categorification of the Jones polynomial agrees with Khovanov's
(\cite{StrSpringer}, \cite{BS1}, \cite{BS3}).

On the other hand,
Cautis and Kamnitzer (\cite{CK}) used the geometry of spaces connected
with two-row Springer fibers to define a related knot homology theory
using certain categories of coherent sheaves, whereas Seidel, Smith
\cite{SS} and Manolescu \cite{Manolescu} constructed a symplectic
version of Khovanov homology using a certain Fukaya category connected
to the Springer fibres of interest to us. Our convolution algebra
construction is motivated by both of these constructions, though
perhaps more strongly the latter. More precisely,
the cohomology of the intersection of two components in our picture
should be seen in analogy to the morphism space between two (compact)
Lagrangian submanifolds in \cite{SS} and one of our main results will
be the definition of a convolution product structure on the direct sum
of all these morphism spaces mimicking the composition of morphisms in
the Fukaya category. \\

We hope that our description of the
convolution algebra will ultimately shed some light on the connection
between the algebraic-representation theoretic categorification and
the geometric ones. In particular, we expect that, with the correct
identifications, the algebras appearing in all three contexts are
isomorphic, establishing some rather surprising equivalences of
categories (see Section \ref{sec:coher-sheav-comp} for results in this direction).\\

An analogous construction associating an algebra to a hypertoric
variety has been developed by the second author with Braden, Licata
and Proudfoot (\cite{BLPW}).  Like the algebra we define, this
hypertoric algebra is quasi-hereditary and moreover Koszul (which is
known to be true for our algebra as well, \cite{BS2}). The Koszul dual
of this hypertoric algebra is the algebra associated via this
convolution construction to the Gale dual hypertoric variety.\\

Let us outline the content of the paper in more detail. For any nilpotent endomorphism $N$ of $\mC^n$, we have the following (in general,
not smooth) subvariety of the full flag variety, which only depends
(up to isomorphism) on the conjugacy class of $N$:
\begin{definition*}
The {\bf Springer fiber} of a nilpotent map $N:\mC^n\to\mC^n$ is the
  variety of all complete flags $\mathcal{F}$ in $\mC^n$ fixed under $N$
  (i.e. for any space $F_i$ of the flag $\mathcal{F}$, we have that
  $NF_i\subset F_{i-1}$).
\end{definition*}
We can always naturally associate a Springer fiber with any parabolic
subalgebra $\p$ of $\mathfrak{sl}_n$ containing the standard Borel of
all upper triangular matrices: given $\p$, we have a composition of
$n$ which, in turn, determines a Jordan type, hence a nilpotent
conjugacy class in $M(n\times n,\mC)$.  More canonically, this is a
regular nilpotent in the Levi of $\p$.  In the present paper, we
restrict to the case where $N$ is nilpotent with two Jordan blocks
(i.e.~where $\p$ is maximal or, equivalently, $\dim\ker N=2$).
\\

In Sections \ref{sec:irred-comp-their}--\ref{sec:pairw-inters-stable},
we will concentrate on combinatorial and geometric preliminaries.  We
first recall the description of irreducible components of these
Springer fibers (following \cite{Fung}), and more generally consider
the closure of cells in the Bia\l ynicki-Birula paving of the Springer
fiber. For all such closures, we verify Fung's conjecture that
pairwise intersections of such are smooth, iterated $\CP^1$-bundles
and explicitly determine their cohomology rings as quotients of the
cohomology ring of the full flag variety.\\

Then, in Section~\ref{sec:convolution-algebras}, we equip the direct
sum of all these cohomologies (with appropriate grading shifts) with a
non-commutative convolution product which turns it into a
finite dimensional graded algebra $H^\bullet$. In the case where the
two Jordan blocks have the same size, the underlying vector space is
isomorphic to the one underlying Khovanov's arc algebra. For all block
lengths, we
obtain the vector spaces underlying the generalized versions of
Khovanov algebras.

The generalized versions of Khovanov's algebra have a
quasi-her\-ed\-itary cover $\tilde{\KH}^\bullet$ described by the
first author in \cite{StrSpringer} and now sometimes called
KS-algebras. For a detailed description and further properties of
these generalized Khovanov algebras and their representation theory,
we refer to \cite{BS1} and \cite{BS2}.  We construct a
quasi-hereditary cover $\tilde{H}^\bullet$ of our first convolution
algebra using a Bia\l ynicki-Birula paving of the Springer fiber, with
respect to a generic cocharacter of the maximal torus commuting with
$N$. The set of fixed points for this torus action are in natural
bijection with the idempotents in the algebra $\KH^\bullet$ (and hence
with indecomposable projective modules in the parabolic category
$\cO_0^\p$ or in the quasi-hereditary cover of the generalized arc
algebra).  We denote by $\st w$ the closures of Bia\l ynicki-Birula
cells. (In our special case these can also be viewed as the stable
manifolds under the Morse flow of the moment map of this cocharacter
or as the closure of a fixed point attracting set).  Taking cohomology
over $\mC$, we show that
\begin{equation*}
 \tilde{H}^\bullet:=\bigoplus_{w,w'}H^\bullet(\st w\cap \st {w'})\langle d(w,w')\rangle
\end{equation*}
can be equipped with a convolution algebra
structure which is a (non-negatively) $\mZ$-graded algebra after appropriate grading shifts $\langle d(w,w')\rangle$ indicated on the
right hand side. \\

We then show the main result of our paper.
\begin{theorem*}
  The algebra $H^\bullet$ (resp. the extended version
  $\tilde{H}^\bullet$) and the generalized arc algebra $\KH^\bullet$
  (resp. its quasi-hereditary cover $\tilde \KH^\bullet$) are
  isomorphic as graded algebras.
\end{theorem*}
Note that by \cite[Theorem 3]{StrSpringer} the category of
$\KH^\bullet$-modules is equivalent to the category of {\it perverse
  sheaves} on a Grassmannian (constructible with respect to the
Schubert stratification). Hence this algebra actually has two
geometric realizations, one arising from constructable sheaves, and
one which seems to be related to the Fukaya category and coherent sheaves.

Since the KS-algebras are the endomorphism rings of certain
projectives in parabolic category $\cO_0^\p$, by Koszul duality
\cite[Theorem 1.1.3]{BGS}, this is isomorphic to an $\Ext$-algebra of
simple modules in a singular block of category $\cO$ corresponding to
a weight precisely fixed by $W_\p$ (in the so-called dot-action). This
theorem then suggests that we have an embedding of this singular
category $\cO$ into the Fukaya category of the Slodowy slice
$\Snk$. In this way one might hope for a direct connection with the
construction in \cite{SS}.
\\

Finally in Section~\ref{sec:coher-sheav-comp}, we consider how our
model (and thus, indirectly, the KS-algebra and category
$\cO^{\mathfrak{p}}_0$) is related to the sheaf-theoretic model of
Khovanov homology given by Cautis and Kamnitzer \cite{CK}. To each
crossing\-less matching $a\in\rm{Cup}(n)$ (that means to each
primitive idempotent in Khovanov's arc algebra) their model associates
a certain coherent sheaf $i_*\Omega(a)^{1/2}$ on a certain compact
smooth variety related with Slo\-do\-wy slices. The variety naturally
contains the Springer fiber we had considered previously, and the
sheaves in question are supported on the component we associated with
$a$. As our notation suggest, these sheaves arise from square roots of
canonical bundles (Theorem~\ref{funcCK}).

We show that, as a vector space, the Ext-algebra of these sheaves can
be identified with our algebra $H^\bullet$ (and thus also with
Khovanov's algebra):

\begin{theorem*}
  With the notation in Section~\ref{sec:coher-sheav-comp} there is an
  isomorphism of graded vector spaces
\begin{equation*}
\Ext^\bullet_{\Coh(\Snk)}(i_*\Omega(a)^{1/2},j_*\Omega(b)^{1/2})\cong  H^\bullet(a\cap b)\langle d(a,b)\rangle,
\end{equation*}
\end{theorem*}

We have not been able to determine whether the Yoneda product on this
space is isomorphic to the arc algebra $\KH^\bullet$.  Obviously, this
would be a very interesting question to resolve. It might be a first
step to solve the question of whether the functorial tangle invariants
of Cautis and Kamnitzer (\cite{CK}) can be identified with the
functorial tangle invariants of Khovanov (\cite{Khotangles}) and
(equivalently) of the second author (\cite{StDuke}).

The half-densities $\Omega(a)^{1/2}$ are simple objects in the heart
$\mathcal{C}$ of a certain $t$-structure on the category of coherent
sheaves on a smooth compact space $Z_n$. This space is a certain compactification of the
pre-image under the Springer resolution of a normal slice to the
nilpotent orbit through $N$ at $N$.

We describe the other simple
objects in this heart, and show that it carries a highest-weight
structure with the same Kazhdan-Lusztig polynomials as the corresponding highest weight KS-algebra, say $A$.

\begin{conjecture*}
  There is an isomorphism between $\mathcal{C}$ and the category of finite
  dimensional modules over the algebra $A$.
\end{conjecture*}

\section*{Acknowledgments}

The authors would like to thank Richard Thomas, Joel Kamnitzer, Daniel
Huybrechts, Tom Braden, Roman Bezrukavnikov, and Clark Barwick for
their insight and suggestions. we are grateful to the referee for many
detailed comments and in particular to Mohammed Abouzaid and Ivan
Smith for pointing out a mistake in a previous version of the
paper. Both authors would like to thank the Institute of Advanced
Study where most of this research was carried out.

\section*{Preliminaries}
In the following, all vector spaces and cohomologies are defined over
$\mC$. We abbreviate $\otimes=\otimes_\mC$. An {\it algebra} will
always be a unitary associative $\mathbb{C}$-algebra. A {\it graded
  vector space} will always be $\mZ$-graded. For a graded vector space
$M$ and $i\in\mZ$ we denote by $M\langle i\rangle$ the graded vector
space with homogeneous components $(M\langle i\rangle)_j=M_{j-i}$.

Let $V$ be an $n$-dimensional complex vector space and $\en:V\to V$ be
a nilpotent endomorphism of Jordan type $(n-k,k)$. For ease, we assume
$2k<n$. Explicitly, we equip $V$ with an ordered basis
$\{p_1,\ldots,p_{n-k},q_1,\ldots, q_{k}\}$ with the action of $\en$
defined by
\begin{equation*} \en(p_i)=p_{i-1},\hspace{.5in}
  \en(q_i)=q_{i-1} \end{equation*}
  where, by convention,
$p_{0}=q_{0}=0$ and often write $\mC^n$ instead of $V$. We let $P=\langle p_1,\ldots,p_{n-k}\rangle$ and
$Q=\langle q_1,\ldots,q_{k}\rangle$.

Let $\flag$ be the variety of complete flags in $V$, and let $\spf$ be the
fixed points of $\exp(\en)$ acting on $\flag$. So, $Y$ consists of all complete flags
$\Fl_0\subset \Fl_1\subset \ldots\subset V$ such that $\en(\Fl_i)\subseteq \Fl_{i-1}$.

The ordering on the basis equips $V$ with a {\it standard flag} $$\{0\}\subset\langle
p_1\rangle\subset \langle p_1,p_2\rangle\subset\cdots \subset\langle p_1,\ldots
p_{n-k}, q_1, \ldots, q_{k-1}\rangle\subset V,$$ which is invariant under $\en.$

\section{Irreducible components and their cohomology}
\label{sec:irred-comp-their}

\subsection{Matchings and tableaux}
\label{sec:matchings-tableaux}

In order to describe the irreducible components of $\spf$, we will
first have to define some combinatorial machinery.  This section will
cover a number of results from the article of Fung \cite{Fung}, which
will be necessary for later.

\begin{definition}
  A {\bf standard tableau} is a filling of the Young diagram of a
  partition such that the rows and columns are strictly decreasing
  (read from the top left corner).
\end{definition}
\begin{definition}
  A {\bf crossingless matching} is a planar diagram consisting of $n$
  points, $k$ cups, and $n-2k$ rays pointing directly downward such
  that each point is attached to exactly one cup or ray, cups only
  pass below points, not above them, and no cup or ray crosses any
  other. We say that a point at the end of a cup is {\bf matched} and
  one at the end of a ray is {\bf orphaned}.
\end{definition}
Given any standard tableau $S$ of shape $(n-k,k)$, we can associate a
crossingless matching $\m(S)$ of $n$ points, numbered from left to the
right, such that the bottom row of the tableau contains all the
numbers which are at the left end of a cup, and the top row of the
diagram contains all the numbers which are at the right endpoint of a
cup, or are the endpoint of a ray.
\begin{proposition}
  This assignment gives in fact a bijection between standard ta\-bleaux
  of shape $(n-k,k)$ and crossingless matchings/cup diagrams
  of $n$ points with $k$ cups and $n-2k$ rays.
\end{proposition}

\begin{ex}
\label{ex1} Let $k=2$, $n=5$. Then we have the following five standard tableaux
$$\young(543,21)\hspace{1.2cm}\young(542,31)\hspace{1.2cm}
\young(532,41)\hspace{1.2cm} \young(531,42)\hspace{1.2cm}\young(541,32)$$ and
the associated cup diagrams (with one orphaned point in each case):

\hspace{1.3cm}
\begin{picture}(0,35)

\put(0,23){$\bullet$}\put(17,23){\oval(10,20)[b]}

\put(10,23){$\bullet$}

\put(20,23){$\bullet$} \put(17,23){\oval(30,30)[b]}

\put(30,23){$\bullet$}

\put(40,23){$\bullet$} \put(42,23){\line(0,-1){15}}
\put(70,23){$\bullet$} \put(77,23){\oval(10,20)[b]}

\put(80,23){$\bullet$}

\put(90,23){$\bullet$} \put(98,23){\oval(10,20)[b]}

\put(100,23){$\bullet$}

\put(110,23){$\bullet$} \put(112,23){\line(0,-1){15}}
\put(140,23){$\bullet$} \put(147,23){\oval(10,20)[b]}

\put(150,23){$\bullet$}

\put(160,23){$\bullet$}

\put(170,23){$\bullet$} \put(177,23){\oval(10,20)[b]}

\put(180,23){$\bullet$} \put(162,23){\line(0,-1){15}}
\put(210,23){$\bullet$}

\put(220,23){$\bullet$} \put(227,23){\oval(10,20)[b]}

\put(230,23){$\bullet$} \put(247,23){\oval(10,20)[b]}

\put(240,23){$\bullet$}

\put(250,23){$\bullet$} \put(212,23){\line(0,-1){15}}
\put(280,23){$\bullet$}

\put(290,23){$\bullet$}\put(307,23){\oval(30,30)[b]}

\put(300,23){$\bullet$} \put(307,23){\oval(10,20)[b]}

\put(310,23){$\bullet$}

\put(320,23){$\bullet$} \put(282,23){\line(0,-1){15}}

\end{picture}
\end{ex}

\begin{ex} The following will be our running example, (and the notation should be kept in mind): Let $k=2$, $n=4$.
Then we have two standard tableaux
$$S(\nested):=\young(43,21)\hspace{1.2cm}S(\nxt):=\young(42,31)$$ where the first
corresponds to the cup diagram $\operatorname{Cup}(\nested)$ with two nested cups,
the second to the cup diagram $\operatorname{Cup}(\nxt)$ with two cups next to
each other. There are no orphaned points.
\end{ex}

Given a tableau $S$ of shape $(n-k,k)$ let $\Sm S$ be set of numbers in the
lower row of the tableau, and $\Sp S$ the set of numbers in the top row. If $S$ is standard,
the cup diagram $\m(S)$ defines a map $\si:\Sm S\to\Sp S$ sending the beginning of
a cup to its end.

Let $\de(i)=(\si(i)-i+1)/2$ be the number of cups nested inside the one connecting $i$ and $\si(i)$ for any $i\in\Sm S$.
Note that $\de(i)$ encodes the size of the cup starting at $i$. For instance, the diagram $\nested$ has $\de(2)=1$ and $\de(1)=2$.
We let $c(i)$ be the column number of $i$, i.e. the number of columns
to the left (inclusive) of the one which $i$ lies in.

\subsection{Components and matchings}
\label{sec:components-matchings}

Spaltenstein \cite{Spa76} and Vargas \cite{Var79} established a
bijection between the irreducible components of $Y$ and the standard
tableaux of shape $(n-k,k)$ which allowed them to describe the
components as closures of explicitly given locally closed subspaces:

\begin{definition} \label{defcom1} Let $S$ be a standard tableau of
  shape $(n-k,k)$. The associated irreducible component $\com S$ is the
  closure of the set of complete flags $\Fl_0\subset \cdots\subset
  \Fl_n=V$ in $Y$ such that for all $i\in\Sm S$, we have $\Fl_i\subseteq
  \Fl_{i-1}+\im N^{c(i)-1}$.
\end{definition}

Alternatively, (see \cite{Fung}) one can use the following much more handy definition: let $t_i$
be the number of indices smaller than or equal to $i$ in the top row, and
similarly for $b_i$ and the bottom row, then we have

\begin{proposition}
\label{defcom} A complete flag
$\{0\}=\Fl_0\subset
  \cdots\subset \Fl_n=V$ lies in $Y_S$ if and only if for all $i\in\Sm S$, we have
  $\en^{\de(i)}(\Fl_{\si(i)})=\Fl_{i-1}$, and for each
  $i\in\Sp S\setminus\si(\Sm S)$, we have $\Fl_{i}=\en^{-b_i}(\im\,\en^{n-k-t_i+b_i})$.
\end{proposition}

Note that the condition of being in a component associated to $S$
means that the spaces $\Fl_i$ where $i$ labels either on orphaned point or the right end of a cup in
$\m(S)$ (i.e. $i\in\Sp S$) as a set, are completely determined by the
spaces $\Fl_{i-1}$ corresponding to the point at the left endpoint of a cup as a set, but this is not true for the ends of an individual cup.

\begin{ex}
\label{irredcpt}
For our running example we have
\begin{align*}
Y_{\nxt}&=\{\Fl_0\subset \Fl_1\subset N^{-1}(\Fl_0)=\langle p_1,q_1\rangle\subset \Fl_3\subset N^{-1}(\Fl_2)=\mC^4\}\subset Y\\
Y_\nested&=\{\Fl_0\subset \Fl_1\subset \Fl_2\subset N^{-1}(\Fl_1)\subset N^{-2}(\Fl_0)=\mC^4\}\subset Y.
\end{align*}
Hence $Y_{\nxt}\cong \mathbb{P}^1\times\mathbb{P}^1$, whereas $Y_\nested$ is a non-trivial $\mathbb{P}^1$-bundle over $\mathbb{P}^1$ (a Hirzebruch surface).
\end{ex}
By \cite[Proposition 5.1]{Fung}, all irreducible components are
iterated $\mathbb{P}^1$-bundles; in particular, they are smooth.

\subsection{Cohomology of components}
\label{sec:cohomology-of-components}
The variety $\flag$ carries $n$ tautological line bundles of the form
$\LB_i=\Fl_i/\Fl_{i-1}$ where we use $\Fl_i$ to denote the corresponding
tautological vector bundle on $\flag$, and its restriction to $\spf$
and $\com S$.  These line bundles generate $\mathrm{Pic}(\flag)$, and
their first Chern classes $x_i=c_1(\LB_i)$ generate the cohomology
ring $H^\bullet(X;\C)$.  This presentation is due to Borel and gives an
isomorphism of $H^\bullet(\flag;\C)$ with the algebra of coinvariants for
the obvious action of the symmetric group $S_n$ on $\C[x_1,\cdots,
x_n]$, that is,
\begin{equation*}
H^\bullet(\flag;\C)\cong \C[x_1,\cdots
x_n]/(\epsilon_1(\Bx),\cdots, \epsilon_n(\Bx))
\end{equation*}
where $\epsilon_i$ is the $i$-th elementary
symmetric polynomial in the variables $x_i$ (see e.g. \cite{Fulton}).
\begin{theorem}
\label{cohomcomp}
 The cohomology ring of $\com S$ has a natural presentation of the form
 \begin{equation*}
   H^\bullet(\com S;\C) \cong \C[\{x_i\}_{i\in\Sm S}]/(\{x_i^2\}_{i\in\Sm S}).
 \end{equation*}
The pullback map $i_S^*:H^\bullet(\flag)\to H^\bullet(\com S)$ is surjective, and given in this presentation by
 \begin{equation*}
   i_S^*(x_i)=
   \begin{cases}
     x_i & i\in \Sm S\\
     -x_{\si^{-1}(i)} & i\in \si(\Sm S)\subset\Sp S\\
     0 &{otherwise}.
   \end{cases}
 \end{equation*}
\end{theorem}
\begin{proof}
  By \cite[Proposition 5.1]{Fung},  $\com S$ an iterated $\CP^1$-bundle, with the
  maps to $\CP^1$ given by the line bundles $\LB_i$ for $i\in \Sm S$. Hence,
  the cohomology ring $H^\bullet(\com S;\C)$ is generated by their first
  Chern classes (since these give a generating set in the associated
  graded with respect to the filtration coming from the Leray-Serre
  spectral sequence).  Since these line bundles are pullbacks from
  $\flag$, the map $i_S^*$ is surjective.

  We will find relations between these using the Chern classes of
  related bundles.  First, note that by the definition of $\com S$,
  we have exact sequences of vector bundles for each $i\in\Sm S$
  \begin{align*}
    0\too \ker N^{\de(i)}\too \Fl_{\si(i)}\overset{\en^{\de(i)}}\too& \Fl_{\si(i)}\too \Fl_{\si(i)}/\Fl_{i-1}\too 0\\
     0\too \ker N^{\de(i)-1}\too \Fl_{\si(i)-1}\overset{\en^{\de(i)-1}}\too& \Fl_{\si(i)-1}\too \Fl_{\si(i)-1}/\Fl_{i}\too 0
  \end{align*}
 It may happen at $\de(i)=1$, in which case we interpret $N^{0}$ as
 the identity map, and the lower exact sequence becomes trivial.
  Since $\ker \en^{\de(i)}$ is a trivial subbundle of $\Fl_n$, we obtain
  in K-theory
  \begin{equation*}
    [\Fl_{\si(i)}/\Fl_{i-1}]-[\Fl_{\si(i)-1}/\Fl_i]=[\LB_{\si(i)}\oplus\LB_{i}]=0.
  \end{equation*}
  The Chern classes of a bundle only depend on its class in
  K-theory, so that the following equalities hold in $H^\bullet(\com S;\C)$:
  \begin{align*}
    c_1(\LB_{\si(i)}\oplus\LB_i)&=x_{\si(i)}+x_{i}=0\\
    c_2(\LB_{\si(i)}\oplus\LB_i)&=x_{\si(i)}x_{i}=0.
  \end{align*}
  If $i\in\Sp S\setminus\si(\Sm S)$, then the bundles $\Fl_i$ and $\Fl_{i-1}$
  are both trivial, so $x_i=0$.

  Thus, the Chern classes $x_i$ for $i\in\Sm S$ generate the cohomology
  of $\com S$, and the relations which we claimed hold.  These must
  be sufficient, since the quotient by the relations we have proven
  above and $H^\bullet(\com S)$ both have dimension $2^k$, the latter by \cite[Theorem 5.3]{Fung}.
\end{proof}

\begin{ex}
Let $\dn\cong \mC[X]/(X^2)$.
We have isomorphisms of graded rings
$$H^\bullet(Y_{C(\nested)})\cong\mC[x_1,x_2]/(x_1^2,x_2^2)\cong\dn\otimes
\dn,$$ and
$$H^\bullet(Y_{C(\nxt)})\cong\mC[x_1,x_3]/(x_1^2,x_3^2)\cong\dn\otimes
\dn.$$
\end{ex}

\section{The Bia\l ynicki-Birula paving and stable manifolds}
\label{sec:torus-acti-irred}

\subsection{The torus action and fixed points}
\label{sec:fixed points}
The torus $(\mathbb{C}^*)^n$ of diagonal matrices in the basis given by the
$p_i$'s and $q_i$'s acts on the flag variety $\flag$ in the natural way and
induces on the Springer fiber $\spf$ an action of a maximal torus of
$Z_G(\en)$. This torus is 2-dimensional, and its action is explicitly given by
$(r,s)\cdot p_i=rp_i, (r,s)\cdot q_i=sq_i$ for $(r,s)\in(\C^*)^2$.\\


 This action has isolated fixed points which we want to
label by row strict tableaux of $(n-k, k)$-shape (i.e. tableaux which are decreasing in the rows, but with no condition on the columns). To any arbitrary row strict tableau $w$ of shape $(n-k,k)$ we associate the
full flag $\fx w$ such that $$\fxi iw=\langle \{p_j,q_r| j\leq t_i, r\leq b_i\}
\rangle,$$ where $t_i$ is the number of indices smaller than or equal to $i$ in
the top row, and similarly for $b_i$ and the bottom row. Note that the standard
flag is of the form $\fx{w^{n-k,k}_{dom}}$, where $w^{n-k,k}_{dom}$ is the row
strict tableau with $1,2,\ldots, n-k$ in the first row; for example
$w^{3,2}_{dom}=\young(321,54)\quad$ and $\;w^{2,2}_{dom}=\young(21,43)$

To any row strict tableau $w$ of shape $(n-k,k)$ we will later associate a
crossingless matching $\m(w)$ of $n$ points by the same rule as before for standard tableaux (but the resulting matching might have in the extreme case only rays and no cups at all); see the paragraph before Theorem~\ref{inclusions} for a precise definition.
There are $\binom{n}{n-k}$ row strict tableaux, which is also the same as the number of fixed points and $\Phi$ defines an explicit bijection:

\begin{lemma}
The map $\Phi: w\mapsto \fx{w}$ defines a bijection between row strict tableaux
of shape $(n-k,k)$ and torus fixed points of $Y$.
\end{lemma}

\begin{proof}
  It is easy to check that $\fx{w}$ is in fact a point in $Y$,
  and obviously a fixed point, since all its component subspaces are
  spanned by weight vectors. The map is $\Phi$ injective by construction.

  On the other hand, if $\mathcal{F}$ is a $T$-fixed flag, then each of
  its constituent subspaces $\Fl_i$ is spanned by the intersections
  $\Fl_i\cap P$ and $\Fl_i\cap Q$.  These, in turn are invariant subspaces
  for $\en|_P$ and $\en|_Q$.  But these restrictions are regular
  nilpotents, so there is a unique invariant subspace of any possible
  dimension, which is of the form $\langle p_1,\ldots, p_i\rangle$ (and
  similarly for $q_j$).  Thus, $\mathcal{F}$ is of the form $\fx{w}$ for some row-strict tableau, and $\Phi$ is surjective.
\end{proof}

Let $w$, $S$ be tableaux of shape $(n-k,k)$, where $w$ is row strict and $S$ is
standard with associated cup diagram $\m(S)$. We consider the sequences ${\bf
a}=a_1 a_2 a_3\ldots a_n$, where $a_i=\wedge$ if $i\in w_\wedge$ and $a_i=\vee$
if $i\in w_\vee$ and call it the {\it weight sequence} of $w$. For instance
$w_{dom}^{3,2}$ has weight sequence $\wedge \wedge\wedge\vee\vee$. (We refer to Example \ref{exattract} for more examples of weight sequences with their cup diagrams.) We can put
the weight sequence on top of the diagram $\m(S)$ and obtain a diagram $w\m(S)$
where the upper ends of each cup or line are decorated with an orientation. We
call $w\m(S)$ {\it oriented} if these decorations induce a well-defined orientation
on $\m(S)$. For instance if $\m(S)$ is one of the cup diagrams from
Example~\ref{ex1} then $w_{dom}^{3,2}\m(S)$ is only oriented if $\m(S)$ is the
last diagram in the list. Note that the number of cups in a cup diagram
$\m(S)$, where $S$ is a standard tableau, is always $k$, hence for any
orientation $w\m(S)$, the decoration at each orphaned vertex will be a
$\wedge$.

\begin{lemma}
\label{fixed-points} A fixed point $\fx{w}$ is in an irreducible component
$\com S$ associated with a cup diagram $C$ if and only if $wC$ is an oriented
cup diagram. In particular, every component contains exactly $2^k$ fixed points.
\end{lemma}

\begin{proof}
Let first $C$ be oriented with the orientation on all cups pointing (counterclockwise) from left to  right (and all
lines pointing up). This is exactly the case when $w=S$ is the standard
tableau associated with $C$. We claim $\fx{w}$ satisfies the conditions
of Proposition~\ref{defcom}. If $i$ is on the top row, then there are exactly
$n-k-c(i)+1$ numbers smaller than or equal to $i$ on the top row, and so $\Fl_i/\Fl_{i-1}$ is
spanned by $p_{t_i}=p_{n-k-c(i)+1}\in \im N^{c(i)-1}$. If $i$ is on the bottom row, then
$\Fl_i/\Fl_{i-1}$ is spanned by $p_{n-c(i)+1}\in\im N^{c(i)-1}$. The claim follows.

Consider now the general case. Let first
$i$ and $\sigma(i)$ be the labels for the two endpoints of a cup. The condition
$\en^{\delta(i)}(\Fl_{\sigma(i)})=\Fl_{i-1}$ is equivalent to exactly half of the
indices between $i$ and $\si(i)$ (inclusive) are contained in $\Sp S$ (or $\Sm S$
respectively). For cups connecting two points next to each other this is
directly equivalent to being oriented. By induction on the length of the cup,
we may assume that each cup between $i$ and $\si(i)$ is oriented. Since there
are no orphaned points below a cup, getting exactly half $\wedge$'s and half
$\vee$, means the labels $i$ and $\sigma(i)$ must carry the opposite
orientations, i.e. the cup is oriented.

Now $i\in\Sp S\backslash\si(\Sm S)$ is the same as saying the point with label $i$
is orphaned. The necessary condition for $\fx w$ only depends on $c(i)$, which
is the same for all $w$ where $wC$ is oriented, because it only depends on the
number of cups and lines to the left of the point $i$. Hence the argument at
the beginning of the proof implies the lemma.
\end{proof}

\begin{ex}
\label{The wis}
There are six row strict tableaux in case $n=4$, $k=2$, hence six fixed points
$w_1, w_2,\ldots, w_6$, corresponding to the six weight sequences
\begin{equation*}
\wedge\wedge\vee\vee, \quad \wedge \vee\wedge\vee, \quad \vee\wedge\wedge\vee,
\quad \wedge\vee\vee\wedge,\quad
\vee\wedge\vee\wedge,\quad\vee\vee\wedge\wedge.
\end{equation*}
The fixed point $w_1$ is the standard flag. Now the component $Y_{S(\nxt)}$
contains $w_i$, $i\in\{2,3,4,5\}$, whereas $Y_{S(\nested)}$ contains the $w_i$,
$i\in\{1,2,5,6\}$.
\end{ex}

\subsection{The paving}
\label{sec:paving}
If we choose a cocharacter $\C^*\hookrightarrow T$ which has the
same fixed points as the whole torus, then we can consider the behavior of
points as $t$ approaches infinity.  We will fix the choice of $t\mapsto
(t^{-1},t)$, that is, subspaces are attracted toward the $q_i$'s as $t$
approaches $\infty$ and towards the $p_i$'s as $t$ approaches $0$.

\begin{definition*}
If $\fx w\in \spf^T$ is a torus fixed point, then we denote the {\bf stable manifold} or {\bf attracting set}
\begin{equation*}
  \st w^0 =\{y\in \spf|\lim_{t\to\infty}t\cdot y=\fx w\},
\end{equation*}
and its closure $\st w=\overline{\st w^0}$.
\end{definition*}

For each flag $\mathcal{F}$ in $Y$, we can obtain a flag $\mathcal{F}'$ (with no longer necessarily distinct spaces) in $P$ by taking the intersections $\tP_i=\Fl_i\cap P$, and similarly in $V/P\cong Q$ given by $\tQ_i=\Fl_i/(\Fl_i\cap P)$.  We can define the new flag $\mathcal{F}'$ by putting $F_i':=\tP_i+\tQ_i\subset P\oplus Q= V$, which is obviously $T$-equivariant.
\begin{proposition}
  A flag $\mathcal{F}$ in $\spf$ is contained in $\st w^0$ if and only if the new flag $\mathcal{F}'$ obtained from it by the procedure above is $\fx w$.
\end{proposition}
\begin{proof}
  Obviously, the new flag $\mathcal{F}'$ does only depend on the orbit of $\mathcal{F}$, in the sense that it does not change if we move inside the torus orbit $O$ containing $\mathcal{F}$. Thus, for any point $\mathcal{G}$ in the closure of $O$ we have $(\mathcal{F}_i\cap P)\subseteq(\mathcal{G}_i\cap P)$ and hence $(\mathcal{F}'_i\cap P)\subseteq(\mathcal{G}'_i\cap P)$ for any $1\leq i\leq n$, since containing a vector is a closed
  condition on a subspace.

  On the other hand, since $P$ has minimal
  weight under $\C^*$, no vector not in $P$ is attracted to $P$ as
  $t\to \infty$, so the size of the intersection with $P$ can only stay the same or
  decrease in that limit.  Thus intersection with $P$ must be fixed
  under the limit.  Since the image in $Q$ has complementary
  dimension, it must also be fixed.
\end{proof}

This makes it clear that $\st w^0$ is
algebraically isomorphic to an affine space (in particular,
diffeomorphic to a disk), since the set of vector spaces projecting
to a given one under a linear map is an affine space.

The structure of these stable manifolds can be understood in terms of
cup diagrams, in much the same way as the structure of the components. To $w$  we attach two (in general different) cup diagrams, $\m(w)$ and $C(w)$ as follows:

For each fixed point $\fx{w}$, there is the diagram $\m(w)$ with the property that $w\m(w)$ has the maximal number of cups amongst all cup diagrams $C$ such that $wC$ is oriented and contains only counter-clockwise cups.  This diagram will have $k_w\leq k$ cups, with equality
$k_w=k$ if and only if $w$ is standard.  One can build this diagram
inductively by adding an arc between any adjacent pair $\vee\wedge$,
and then continuing the process for the sequence with these points
excluded. We then add lines to the remaining points. We call $\m(w)$ {\it the cup diagram associated with $w$}.

Rather than adding these lines, we could complete to an oriented cup
diagram $C(w)$ with $k$ cups, by matching all the $\vee$'s in the only possible way. Call the corresponding
standard tableau $\St w$.

\begin{theorem}
\label{inclusions}
Let $w$ be a row strict tableau. Then $\st w$ is
   the subset of $\com {\St w}$ containing exactly the flags which satisfy the additional property: if $i\in w_{\wedge}\cap \St w_{\vee}$, then $F_i$ coincides with the $i$th subspace of the fixed point $\fx w$.

In particular, for any standard tableau $S$, we have $\st S=\com S$.
\end{theorem}
\begin{proof}
First we confirm that these relations hold on $\st w^0$ (and thus on $\st w$, since they are closed conditions).

Let $\mathcal{F}$ be a point in $\st w^0$. Let us first assume there is at least one cup in $\m(w)$, so in particular a minimal one not containing any other cup. This means there is some index $i\in w_\vee$ with $\de(i)=1$. The result we desire is that $\Fl_{i+1}=\en^{-1}(\Fl_{i-1})$.

  First, note that since the index $i+1$ is marked with an $\wedge$ in $w$,
  then we must have \begin{equation*}
\Fl_{i+1}\supset \Fl_{i}+\en_P^{-1}(\Fl_{i}\cap
  P).    \end{equation*}

On the other hand, since $i$ is marked with an $\up$, we must have  $\Fl_{i}\cap
  P = \Fl_{i-1}\cap
  P$ and it follows
\begin{equation*}
\en^{-1}(\Fl_{i-1})\supset \Fl_{i}+\en_P^{-1}(\Fl_{i-1}\cap  P) = \Fl_{i}+\en_P^{-1}(\Fl_{i}\cap P)
\end{equation*}
All of these spaces are of dimension $i+1$, so we must have $\Fl_{i+1}=\en^{-1}(\Fl_{i-1})$.

Let $w'$ denote $w$ with $i,i+1$ removed. Applying $\en$ to all spaces of dimension bigger than $i+1$ provides a map $q_i: \st w^0\to \st {w'}^0$ which extends to a map $q_i:\st w\to \st {w'}$ between the closures.  The relation for a cup in $S(w')$ pulls back to that for the corresponding cup of $S(w)$.

Thus, by induction, we may reduce to the case where there are no cups in $\m(w)$ (that is, $w$ is a series of $\wedge$'s followed by $\vee's$).  In this case, our claim simply reduces to the claim that $\st w=\{\fx w\}$. This is indeed the case, since for any index in $w_{\wedge}$, we must have $F_i\subset P$, and $\en$ acts regularly on $P$ so all $\en$-invariant subspaces are also $T$-equivariant.  Similarly, for any $i\in w_{\vee}$, we must have $F_i\supset P$, and $\en$ acts regularly on $V/P\cong Q$. Therefore $\mathcal{F}$ satisfies the required relations.

On the other hand $\fx w$ obviously satisfies the conditions coming from cups in $C(w)$, and our requirement on elements of $w_{\wedge}\cap \St w_{\vee}$, and any flag satisfying these relations is in the closure of $\st w^0$.
\end{proof}

\begin{ex}
\label{exattract}
The cup diagrams $\m(w_i)$ associated to the weights $w_i$, $1\leq i\leq 6$ from Example~\ref{The wis} are as follows:\\

\begin{picture}(20,35)
\put(0,23){$\bullet$}\put(2,23){\line(0,-1){15}}

\put(10,23){$\bullet$}\put(12,23){\line(0,-1){15}}

\put(20,23){$\bullet$}\put(22,23){\line(0,-1){15}}

\put(30,23){$\bullet$}\put(32,23){\line(0,-1){15}}

\put(60,23){$\bullet$} \put(62,23){\line(0,-1){15}}

\put(70,23){$\bullet$} \put(77,23){\oval(10,20)[b]}

\put(80,23){$\bullet$}

\put(90,23){$\bullet$} \put(92,23){\line(0,-1){15}}

\put(120,23){$\bullet$} \put(127,23){\oval(10,20)[b]}

\put(130,23){$\bullet$}

\put(140,23){$\bullet$} \put(142,23){\line(0,-1){15}}

\put(150,23){$\bullet$} \put(152,23){\line(0,-1){15}}

\put(180,23){$\bullet$} \put(182,23){\line(0,-1){15}}

\put(190,23){$\bullet$} \put(192,23){\line(0,-1){15}}

\put(200,23){$\bullet$} \put(207,23){\oval(10,20)[b]}

\put(210,23){$\bullet$}

\put(240,23){$\bullet$}\put(247,23){\oval(10,20)[b]}

\put(250,23){$\bullet$}

\put(260,23){$\bullet$} \put(267,23){\oval(10,20)[b]}

\put(270,23){$\bullet$}

\put(300,23){$\bullet$}\put(317,23){\oval(30,30)[b]}

\put(310,23){$\bullet$}\put(317,23){\oval(10,20)[b]}

\put(320,23){$\bullet$}

\put(330,23){$\bullet$}
\end{picture}

On the other hand, the cup diagrams $C(w_i)$ for the weights $w_i$, $1\leq i\leq 6$ are as follows:\\

\begin{picture}(20,35)
\put(0,23){$\bullet$}\put(17,23){\oval(30,30)[b]}

\put(10,23){$\bullet$}\put(17,23){\oval(10,20)[b]}

\put(20,23){$\bullet$}\put(22,23)

\put(30,23){$\bullet$}\put(32,23)

\put(60,23){$\bullet$} \put(77,23){\oval(30,30)[b]}

\put(70,23){$\bullet$} \put(77,23){\oval(10,20)[b]}

\put(80,23){$\bullet$}

\put(90,23){$\bullet$}

\put(120,23){$\bullet$} \put(127,23){\oval(10,20)[b]}

\put(130,23){$\bullet$}

\put(140,23){$\bullet$} \put(147,23){\oval(10,20)[b]}

\put(150,23){$\bullet$}

\put(180,23){$\bullet$} \put(187,23){\oval(10,20)[b]}

\put(190,23){$\bullet$}

\put(200,23){$\bullet$} \put(207,23){\oval(10,20)[b]}

\put(210,23){$\bullet$}

\put(240,23){$\bullet$}\put(247,23){\oval(10,20)[b]}

\put(250,23){$\bullet$}

\put(260,23){$\bullet$} \put(267,23){\oval(10,20)[b]}

\put(270,23){$\bullet$}

\put(300,23){$\bullet$}\put(317,23){\oval(30,30)[b]}

\put(310,23){$\bullet$}\put(317,23){\oval(10,20)[b]}

\put(320,23){$\bullet$}

\put(330,23){$\bullet$}
\end{picture}

There are the two irreducible components from Example \ref{irredcpt}
\begin{align*}
\st{w_5}&=\{\Fl_0\subset \Fl_1\subset N^{-1}(\Fl_0)=\langle p_1,q_1\rangle\subset \Fl_3\subset N^{-1}(\Fl_2)=\mC^4\}\subset Y\\
\st{w_6}&=\{\Fl_0\subset \Fl_1\subset \Fl_2\subset N^{-1}(\Fl_1)\subset N^{-2}(\Fl_0)=\mC^4\}\subset Y
\end{align*}
and the additional stable manifolds
\begin{align*}
\st{w_4}&=\{\Fl_0\subset \langle p_1\rangle\subset \langle p_1,q_1\rangle\subset \Fl_3\subset\mC^4\}\subset \st{w_5},\\
\st{w_3}&=
\{\Fl_0\subset \Fl_1\subset\langle p_1,q_1\rangle\subset
\langle p_1,p_2,q_1\rangle\subset\mC^4\}\subset\st{w_5}\\
\st{w_2}&=
\{\Fl_0\subset \langle p_1\rangle\subset \Fl_2\subset
\langle p_1,p_2,q_1\rangle\subset\mC^4\}\subset\st{w_6},\\
\st{w_1}&=
\{\Fl_0\subset \langle p_1\rangle\subset \langle p_1,p_2\rangle\subset
\langle p_1,p_2,q_1\rangle\subset\mC^4\}=\{w_1\}\subset\st{w_6}.
\end{align*}
\end{ex}

\subsection{The cohomology of stable manifolds}
\label{sec:stable-manifolds}
The proof of Theorem~\ref{inclusions} with the result of Theorem~\ref{cohomcomp} enables us to calculate the cohomology
of the stable manifolds $\st w$. Let $\Sm \m(w)$ (resp.  $\si(\Sm \m(w))$) be the set of indices of vertices which are at the left (resp, right) end of a cup in
$\m(w)$ (i.e. those at which we have a free choice, and are not constrained to
match the fixed point).
\begin{theorem}
\label{cohomattract}
   The cohomology ring of $\st w$ has a natural presentation of the form
 \begin{equation*}
   H^\bullet(\st w;\C) \cong \C[\{x_i\}_{i\in\Sm \m(w)}]/(\{x_i^2\}_{i\in\Sm \m(w)})
 \end{equation*}
 with the surjective pullback map $i_w^*:H^\bullet(\flag)\to H^\bullet(\st w)$ given in this presentation by
 \begin{equation}
 \label{equ}
   i_S^*(x_i)=
   \begin{cases}
     x_i & i\in \Sm \m(w)\\
     -x_{\si^{-1}(i)} & i\in \si(\Sm \m(w))\\
     0 & \text{otherwise}
   \end{cases}
 \end{equation}
\end{theorem}

\begin{proof}
Theorem~\ref{inclusions} implies that the map in question is surjective and gives the last case in \eqref{equ}. The second relation has to hold because of Theorem~\ref{inclusions} together with Theorem~\ref{cohomcomp}. Finally, the proof of Theorem~\ref{inclusions} implies that the dimension of $H^\bullet(\st w;\C)$ equals $2^a$, where $a$ is the number of cups in $\m(w)$, hence there are not more relations and the statement follows.
\end{proof}

\begin{ex}
In the situation of Example \ref{The wis} we have isomorphisms as follows:
\begin{equation*}
\begin{array}[tb]{lclclclclc}
H^\bullet(\st{w_1})&\cong&\mC,\\
H^\bullet(\st{w_2})&\cong&\mC[x_2]/(x_2^2)&\cong&\dn,\\
H^\bullet(\st{w_3})&\cong&\mC[x_1]/(x_1^2)&\cong&\dn,\\
H^\bullet(\st{w_4})&\cong&\mC[x_3]/(x_3^2)&\cong&\dn,\\
H^\bullet(\st{w_5})&\cong&\mC[x_1,x_3]/(x_1^2,x_3^2)&\cong&\dn\otimes\dn,\\
H^\bullet(\st{w_6})&\cong&\mC[x_1,x_2]/(x_1^2,x_2^2)&\cong&\dn\otimes\dn.
\end{array}
\end{equation*}
\end{ex}

\section{Pairwise intersections of stable manifolds}
\label{sec:pairw-inters-stable}

\subsection{Fixed points of intersections}
\label{sec:fixed-intersections}
The first step in understanding the structure of the intersection of
stable manifolds is to calculate the torus fixed points which lie in
the intersection.

Let $w$ and $w'$ be two row-strict tableaux of shape $(n-k, k)$ with associated
cup diagrams $C=\m(w)$ and $D=\m(w')$. Let $\overline{D}C$ be the diagram obtained by taking $D$, reflecting it in the
horizontal line containing the dots and putting it on top of the diagram $C$, identifying the points with the same label.
The result will be (up to homotopy) a collection of lines and circles.

\begin{definition*}
An {\bf orientation} of $\overline{D}C$ or
$\overline{C}D$ is a row strict tableau $v$ such that $vD$ and $vC$
are oriented.  In particular, this requires the weight sequence for $v$ to match the one for $w$ at any
unmatched points in $C$, and the one for $w'$ at any unmatched points in $D$.
\end{definition*}

\begin{lemma}
\label{combintersect}
  Let $\st w, \st {w'}$ be stable manifolds in $Y$ with associated cup
  diagrams $C$ and $D$. Then the number of fixed points contained in
  $\st w\cap \st {w'}$ equals the number of orientations of the
  diagram $\overline{D}C$.  In particular, the number of fixed points is either
\begin{itemize}
  \item  zero (if there is at least one line
  where the orientations required by orphaned points are incompatible),
  \item  one (if all lines are oriented and there are no circles),
  \item or $2^c$ (otherwise), where $c$ is the number of
  circles in $\overline{D}C$.
\end{itemize}
\end{lemma}

\begin{proof}
  By Lemma~\ref{fixed-points} the number of fixed points in the
  intersection of two irreducible components is the number of weight sequences which give rise to
  an orientation of $C$ and $D$ at the same time, and hence to an
  orientation of $\overline{D}C$. By Theorem \ref{inclusions} this is true more generally for intersections of two stable manifolds.
  For each circle there are exactly
  two such choices of an orientation and for each line there is a
  unique orientation. There is no orientation if the endpoints of some
  line are contained in the same cup diagram. The statement follows.
\end{proof}

\begin{corollary}\label{non-empty-intersection}
  The intersection $\st w\cap\st{w'}$ is
  \begin{itemize}
  \item non-empty if and only if
  $\overline{D}C$ has an orientation,
  \item a single point if and only
  if there is a unique such orientation.
  \end{itemize}
\end{corollary}
\begin{proof}
The intersection $\st w\cap\st{w'}$ is projective, and so it is either empty or has a fixed point
by Borel's fixed point theorem. Moreover, if $\st w\cap\st{w'}$ contains a point $x$ which is not a fixed
  point, then the limits $\lim_{t\to 0}t\cdot x$ and $\lim_{t\to
    \infty}t\cdot x$ exist and are different torus
  fixed points, since they have different moment map images.
\end{proof}

\begin{ex}
Using the cup diagram in Example~\ref{exattract} one easily obtains the following three sets
telling when the intersection $\st
{w_i}\cap\st{w_j}$ is empty, contains exactly one fixed point, or contains
exactly two fixed points respectively:
\begin{align*}
(i,j)&\in\{(1,3),(1,4),(1,5)\},\\
(i,j)&\in\{(1,2),(1,6),(2,3),(2,4),(2,5),(3,4),(3,6),(4,6)\}\\
(i,j)&\in\{(2,6),(3,5),(4,5),(5,6)\}.
\end{align*}
\end{ex}

\subsection{Structure of intersections}
To fully describe the structure of the intersections, we will require a bit
more machinery.  We first restate once more the condition for a flag
$\mathcal{F}_\bullet\in Y$ being contained in an irreducible component
$\com S$. We introduce an equivalence relation on the subspaces $F_i$
of $\mathcal{F}$ by grouping them into classes such that fixing one
element of a given class determines our choice of all the others.
Consider the cup diagram $C$ associated to $S$, and let $\sip(a)=\si(a+1)$.

\begin{definition}
  Let $i\re j$ (or more precisely $i\re_C j$) be the equivalence
  relation on the set $\{1,2,\ldots n\}$ obtained by taking the
  transitive closure of the reflexive and symmetric relations $i=j$,
  or $\sip(i)=j$ or $\sip(j)=i$ (when $\sip$ is defined). Note that
  the set of minimal representatives of the equivalence classes equals
  $\Sm S$.
\end{definition}

For all $i,j$ such that $\sip(i)=j$ or vice versa, we have
$\Fl_i=\en^{(j-i)/2}(\Fl_j)$ for any flag $\mathcal{F}\in \com S$. Since
this condition is transitive, we obtain that whenever $i\re j$, we
have $\Fl_i=\en^{(j-i)/2}(\Fl_j)$, and along with attaching a fixed
subspace to each orphaned vertex, this is a full set of relations for
$\com S$.

\begin{proposition}\label{cup-rel}
  If $b=\si(a)$, then $b\re a-1$ and $a\re b-1$.
\end{proposition}
\begin{proof}
  The first relation is by definition.  To get the second, note that
  $\sip(a)<b$ (by the non-crossing condition), and either $\sip(a)=b-1$
  (in which case we obtain the desired equivalence), or
  $a<\sip(a)<\sip^2(a)<b$ (again, by non-crossing). Since there are
  finitely many indices between $a$ and $b$, we must have
  $b-1=\sip^\ell(a)$ for some $\ell$, and so $a\re b-1$.
\end{proof}

Now we introduce distinguished representatives for each equivalence class:
\begin{definition}
  Given two row strict tableaux $w$ and $w'$ with associated cup diagrams $C$ and $D$, we let $i\approx j$ (or more precisely
  $i\approx_{C,D}\! j$ or  $i\approx_{w,w'}\! j$) be the transitive closure of the relations of
  the form $i\re_Cj$ {\em and} those of the form $i\re_Dj$. We let
$\mE(C,D)$ or $\mE(w,w')$ denote the set of minimal representatives for $\approx_{C,D}$ with the
subset $\cE(C,D)=\cE(w,w')$ given by all points lying on a circle in $\overline{D}C$.
\end{definition}

\begin{ex}
The equivalence classes for our running example are
\begin{equation*}
\sim_{S(\nxt)}:\quad\{0,2,4\},\; \{1\}, \{3\},\hspace{1in}
\sim_{S(\nested)}:\quad\{0,4\},\;\{1,3\},\{2\}.
\end{equation*}
There are two equivalence classes for
$\approx_{{S(\nxt)},{S(\nested)}}$, namely $\{0,2,4\}$ and $\{1,3\}$. The set of minimal representatives are $$\mE({S(\nxt)},{S(\nested)})=\{0,1\}\quad\text{     and     }\quad\cE({S(\nxt)},{S(\nested)})=\{1\}.$$
\end{ex}

\begin{ex}
We denote by $S_1, S_2, \ldots, S_5$ the five standard tableaux of
Example~\ref{ex1}. The set of equivalence classes of $\sim_{S_i}$ are the
following:
\begin{align*}
\sim_{S_1}:\quad&\{\{0,4\},\{1,3\},\{2\},\{5\}\},
&\sim_{S_2}:\quad&\{\{0,2,4\},\{1\},\{3\},\{5\}\},\\
\sim_{S_3}:\quad&\{\{0,2\},\{1\},\{3,5\},\{4\}\},
&\sim_{S_4}:\quad&\{\{0\},\{1,3,5\},\{2\},\{4\}\},\\
\sim_{S_5}:\quad&\{\{0\},\{1,5\},\{2,4\},\{3\}\}
\end{align*}

Now the equivalence classes for $\approx_{{S_1},{S_4}}$ are for instance
$$\{\{0,4\},\{1,3,5\},\{2\}\}$$ with $\mE(S_1,S_4)=\{0,1,2\}$ and $\cE(S_1,S_4)=\{2\}$, since $1$ labels a point on
a line, whereas $2$ labels a point on a circle. The flags contained in $\com
{S_1}\cap \com {S_4}$ are exactly the flags in $Y$ of the form
\begin{equation*}
\{0\}\subset \im N^{2}\subset \Fl_2\subset N^{-1}(\Fl_1)\subset
N^{-2}(\{0\})\subset N^{-2}(\Fl_1)=\mC^5.
\end{equation*}
\end{ex}

\begin{theorem}
\label{intersectcomb}
   The set $\mE(D,C)$ of minimal representatives of the
  equivalence classes contains, apart from zero, exactly the left most points in either a circle
  or line of $\overline{D}C$.
\end{theorem}

\begin{proof}
  Indeed, let $a,b,c$ be the labels of three points in $\overline{D}C$
  such that $a$ and $b$ are connected via a cup and $b$ and $c$ via a
  cap. By Proposition~\ref{cup-rel}, we have $c\re_D b-1\re_C a$, so
  $c\approx_{C,D}a$.

  Repeating this argument implies the following: If two points on a
  circle in $\overline{D}C$ are joined by a path with an even
  number of arcs, then they are equivalent.  Thus all indices on any
  circle are equivalent either to its leftmost point $p$, or to a
  point adjacent to $p$ by a single arc.  Applying
  Proposition~\ref{cup-rel} again, this shows that each point in the
  circle is equivalent to $p$ or $p-1$, the latter of which must be
  equivalent to the leftmost point in another circle or to $0$, by
  induction.
\end{proof}

We note that the set $\mE(D,C)$ can be equipped with a partial order
defined by $a\geq b$ if the circle $a$ lies on is nested inside that
$b$ lies on. Hence outer circles are minimal in this ordering.
 This poset has a natural rank function
$r:\mE(D,C)\to\mZ$ given by 0 on all lines, 1 on all circles not
nested inside any other, and thereafter increasing with the depth of
nesting.  Recall that a flag indexed by a ranked poset is a map of
ranked posets from that poset to the ranked poset of subspaces of a
given vector space.

The equivalence relation $\approx$ allows us to prove Conjecture 7.1 of \cite{Fung}:

\begin{theorem}
\label{intersection-fiber}
    The variety $\st w\cap\st {w'}$ is canonically isomorphic to the space of flags indexed by the ranked poset $\mE(D,C)$ invariant under $N$. In particular,
  \begin{enumerate}
\item  $\st w\cap\st {w'}$ is an iterated fiber
  bundle of base type $(\CP^1,\CP^1,\ldots, \CP^1)$ where the numbers
  of terms is the number $c$ of closed circles in $\overline{D}C$ (if
  $c=0$, the intersection is a point),
\item $\st w\cap\st {w'}$ is smooth, and
\item $H^\bullet(\st w\cap\st {w'})\cong
  \dn^{\otimes c}$ as graded vector spaces.
\end{enumerate}
\end{theorem}
\begin{proof}
The consequences are clear, hence we only prove the first statement.
  Since any comparable circles are on the same side of each line, we
  can divide our poset into subsets consisting of the circles between
  any adjacent lines.  The space of flags indexed by this sub-poset in
  $V$ is isomorphic to space of such flags in $V/\Fl_{\ell(a)}$, where
  for $a\in \mE(D,C)$, we let $\ell(a)$ be the left-most point on the
  right-most line that $a$ lies on the right side of, and thus our
  claim is that our intersection is isomorphic to the product of these
  spaces of flags.

  Consider the subspaces
  $G_a=N^{(a-r(a)+\ell(a))/2}(\Fl_a)/\Fl_{\ell(a)}$.  This is a
  subspace of $V/\Fl_{\ell(a)}$ of dimension $r(a)$.

  If $a\geq b$, and $r(a)=r(b)+1$, then we have $a-1 \approx b$, since
  either $a-1=b$, or $a-1$ lies on a circle with leftmost point $a'$.
  Since $a'\geq b$, we have $r(a') > r(b)$, so $a\ngeq a'$.  Thus, we
  have $a-1\approx a'-1$, and by induction, our claim follows.  Thus
  $G_a\supset N^{(a-r(a)+\ell(a))/2}(\Fl_{a-1})/\Fl_{\ell(a)}=G_b$,
  since
  \begin{equation*}
    a-r(a)=((a-1)-b)+(b-r(b))
  \end{equation*}
  and $\ell(a)=\ell(b)$.

  By induction, this establishes that $G_a$ is indeed a flag over our
  poset.

  Conversely, we can define an element of our intersection, given such
  a flag, by defining $\Fl_i$ by
  $N^{-(i-r(i'))/2}(G_{i'}+\Fl_{\ell(i)})$ where $i'$ is the
  representative of $i$ in $\mE(D,C)$.

  This variety is an iterated $\CP^1$-bundle, since forgetting the
  vector space attached to a maximal element $a$ obviously defines a
  map to the set of flags indexed by a poset with this point removed.
  This map is surjective, since the interval below $a$ is a chain, so
  the space attached to it can be chosen in increasing order. On the
  other hand, the fiber of this map is $\CP(N^{-1}(G_{a'})/G_{a'})$
  for $a'$ the unique element that $a$ covers in this poset (the
  circle immediately containing it).  This is a $\CP^1$, since
  $G_{a'}\subset \im N$, for any $a'\neq a$ for simple dimensional
  reasons (we must have $r(a') < r(a) \leq k$ since no diagram can
  have have more than $k$ circles, and thus no more than $1$ of rank
  $k$). This is thus a general result for flags indexed by any poset
  where all intervals are chains, and the rank is bounded by $k$.
\end{proof}
This theorem has a natural generalization to intersections of
arbitrary finitely many numbers of $\st {w_i}$, given by a rank function on the set
of equivalence classes of the relation generated by $\sim_{w_i}$ for all
$i$.  Let $\mE(w_1,\cdots, w_n)$ be the set of minimal elements of
these equivalence classes.  This can be defined inductively by the following rule:
\begin{itemize}
\item If an equivalence class contains a line, $r(i)=0$.
\item If $i\in\mE(w_1,\ldots,w_n)$, and $j\in\mE(w_1,\ldots,w_n)$ is
  the minimal representative of $i-1$, then $$r(i)=
  \begin{cases}
    r(i-1)+1 & \text{if $i-1\equiv j\pmod 2$}\\
    r(i-1)   & \text{if $i-1\not\equiv j\pmod 2$}
  \end{cases}
$$
\end{itemize}

This rank function will be of great importance in the next section.

\subsection{The cohomology of pairwise intersections as bimodules}
\label{sec:cohom-pairw-inters}

Theorem~\ref{intersection-fiber} enables us to calculate the cohomology $H^\bullet(\st w\cap\st {w'})$ of the
intersection of two stable manifolds as a module over the cohomology of
$H^\bullet(\flag)$, and thus as a $(H^\bullet(\st w),H^\bullet(\st {w'}))$-bimodule.

For any $1\leq i,j\leq n$ we set $\ep(i,j)=0$ if $i$ and $j$ are not on the same circle in $\overline{\m(w)}\m(w')$, and  $\ep(i,j)_{w,w'}=\ep(i,j)=(-1)^{a}$ if $i$ and $j$ lie on the same circle with $a$ being the number of arcs in a path between them. Note that, although $a$ depends on the chosen path, the number $(-1)^a$ does not.
\begin{theorem}
\label{intersections}
Assume the intersection $\st w\cap\st {w'}$ is non-empty. Then the cohomology ring $H^\bullet(\st w\cap\st {w'})$ has the  presentation
  \begin{equation}\label{intersection-equality}
    H^\bullet(\st w\cap\st {w'})=\C\left[\{x_i\}\right]/(\{x_i^2\}),
  \end{equation}
  where the index $i$ runs through $\cE(\m(w),\m(w'))$. The pullback map $$i_{w,{w'}}^*:H^\bullet(\flag)\to H^\bullet(\st
  w\cap\st{w'})$$ is surjective and given by
  \begin{equation*}
  i_{w,{w'}}^*(x_i)=\sum_{j\in\cE(w,w')}\epsilon(i,j)x_j
  \end{equation*}
  In particular, the image of $x_i$ is zero if and only if $i$ does not lie on a closed circle.
\end{theorem}
\begin{proof}
  By Theorem~\ref{cohomattract} we know in particular
\begin{math}
\ker i_{w,{w'}}^*\supseteq\ker i_w^*+\ker i_{w'}^*.
\end{math}
Hence there is a well-defined map
\begin{equation*}
\psi:\quad H^\bullet(\flag)/(\ker i_w^*+\ker i_{w'}^*)\to H^\bullet(\st w\cap\st w').
\end{equation*}
By Theorem~\ref{intersection-fiber}, $\psi$ is surjective since the cohomology
of the intersection is generated in degree two. Comparing dimensions (Theorem~\ref{cohomattract} provides the dimension of the left hand side whereas Corollary~\ref{intersection-fiber} gives the dimension of the right hand side), we see $\psi$ must be an
isomorphism.
\end{proof}

\begin{ex} The only interesting cases for $\st {w_i}\cap\st {w_j}$ where $i\not=j$ (notation as in Example~\ref{The wis}) are
\begin{align*}
H^\bullet(\st {w_2}\cap\st {w_6})\cong\mC[x_2]/(x_2^2),&& H^\bullet(\st {w_3}\cap\st
{w_5})\cong\mC[x_1]/(x_1^2), \\
H^\bullet(\st {w_4}\cap\st {w_5})\cong\mC[x_3]/(x_3^2),&& H^\bullet(\st {w_5}\cap\st
{w_6})\cong\mC[x_1]/(x_1^2),
\end{align*}
since in all other cases where the intersection is non-trivial we get $\mC$.
\end{ex}

Similar bimodules have appeared previously: first in work of Khovanov
\cite{KhoJones}, \cite{Khotangles} in the case $2k=n$, for pairs of
standard tableaux; then in the general case in work of the first
author \cite{StrSpringer} and \cite{BS1}. Our construction agrees with
the the latter two, and so the cohomology rings of stable manifolds
$\st w$ are naturally isomorphic to the endomorphism ring of the
indecomposable projective module corresponding to $\m(w)$ for the
algebra denoted $\ST^{n-k,k}$ of \cite{StrSpringer}, \cite{BS1}. The
category of modules over this algebra is equivalent to the category of
perverse sheaves on the Grassmannian of $k$-planes in $\mC^n$ (see
\cite{StrSpringer}) and related to the representation theory (the
so-called category $\cO$) of the general Lie algebra
$\mathfrak{gl}(n,\mC)$.

\subsection{Background from category $\cO$}
\label{sec:category-O}

Let us briefly recall the construction of \cite{StrSpringer} and the connection to (parabolic) category $\cO$.  For details on and properties of category $\cO$ and its parabolic version see for example \cite{BGG} and \cite{RC}, or the recent book \cite[Chapter 9]{HumphreysO}).\\

The symmetric group $S_n$ acts (from the right) on the set $W(n-k,k)$
of weight diagrams with $n-k$ $\wedge$'s and $k$ $\vee$'s by
permutation. The stabilizer of the weight
$w_{dom}=\wedge\wedge\ldots\wedge\vee\ldots\vee$ is the Young subgroup
$S_{n-k}\times S_k$ of $S_n$. Hence we get a bijection between the set
$W^{n-k,k}$ of shortest coset representatives $S_{n-k}\times
S_k\backslash S_n$ and the set $W(n-k,k)$ under which $w_{dom}$
corresponds to the identity element in $S_n$. On the other hand, the
set $W^{n-k,k}$ labels in a natural way also the simple modules in the
principal block $\cO^{n-k,k}_0$ of the parabolic category
$\cO^{n-k,k}$ for the Lie algebra $\mathfrak{gl}(n,\mC)$.

These simple modules are exactly the simple highest weight modules $L(x\cdot0)$
in the principal block of $\cO$ for $\mathfrak{gl}(n,\mC))$ which are locally finite with respect to the
parabolic $\mathfrak{p}=\mb+\mathfrak{l}$, where $\mb$ is the standard
Borel given by upper triangular matrices and
$\mathfrak{l}\cong\mathfrak{gl}({n-k},\mC)\times\mathfrak{gl}({k},\mC)$
is the subalgebra of $\mathfrak{gl}(n,\mC)$ given by all
$(n-k,k)$-block matrices. Let $P(x\cdot0)\in\cO^{n-k,k}_0$ be the
(indecomposable) projective cover of $L(x\cdot0)$. We have now set up
a bijection between the indecomposable modules $P(x\cdot0)$ and the
stable manifolds $\st w$ by mapping $P(x\cdot0)$ to its weight diagram
in $W(n-k,k)$ which then in turn is associated with some row strict
tableau $w=w(x)$ determining the stable manifold $\st w=\st
{w(x)}$. Let $\operatorname{Cup}(x)$ be the corresponding cup diagram.\\

The endomorphism algebra $\ST^{n-k,k}$ of a minimal projective
generator $\bigoplus_x P(x\cdot0)$ in $\cO^{n-k,k}_0$ has the
following description: Let $x$, $y\in W^{n-k,k}$. Then
$\Hom_\cO(P(x\cdot0),P(y\cdot0))=\{0\}$ in case the diagram
$\overline{\operatorname{Cup}(x)}\operatorname{Cup}(y)$ cannot be
oriented. Otherwise there is an isomorphism of vector spaces
 \begin{equation*}
 \Hom_\cO(P(x\cdot0),P(y\cdot0))=\dn^{\otimes c(x,y)},
 \end{equation*}
 where $c(x,y)$ is the number of circles in the diagram $\overline{\operatorname{Cup}(x)}\operatorname{Cup}(y)$ (with $\dn^{\otimes0}=\C$ by definition). In particular, thanks to Theorem~\ref{intersection-fiber},
 $$\Hom_\cO(P(x\cdot0),P(y\cdot0))\cong H^\bullet(\st {w(x)}\cap \st {w(y)})$$ as vector spaces.

 The endomorphism algebra $\ST^{n-k,k}$ can be equipped with a Koszul
 grading (\cite{BGS} or \cite[Theorem 5.6]{BS2} and \cite[Theorem 1.1]{BS3} for a more elementary proof).
 Let $\tilde{P}(x\cdot0)$ be the standard
 graded lift of $P(x\cdot0)$. This is a graded $\ST^{n-k,k}$-module
 whose head is concentrated in degree zero and which is isomorphic to
 $P(x\cdot0)$ after forgetting the grading. Since $P(x\cdot0)$ is
 indecomposable, such a standard graded lift is unique up to
 isomorphism (\cite[Lemma 2.5.3]{BGS}). Then the space
 $\Hom_{\ST^{n-k,k}}(\tilde P(x\cdot0),\tilde P(y\cdot0))$ is a graded
 vector space isomorphic to
 \begin{equation}
 \label{dshift}
 H^\bullet(\st {w(x)}\cap \st {w(y)})\langle d(x,y)\rangle,
 \end{equation}
 where $d(x,y)=k-c(x,y)$. (Since $c(x,y)$ is the dimension of $\st {w(x)}\cap \st {w(y)}$, the shift encodes its codimension in a Lagrangian in which it is contained.)
 In particular, $\End_{\ST^{n-k,k}}(\tilde P(x\cdot0))\cong H^\bullet(\st {w(x)})$. The multiplication in $\ST^{n-k,k}$ was defined using a TQFT-procedure generalizing Khovanov's (see \cite{StrSpringer}, \cite{BS1}, \cite{KhoJones}). From the definitions it follows in particular, $$\End_{\ST^{n-k,k}}(\tilde P(x\cdot0))\cong H^\bullet(\st {w(x)})$$ as graded algebras.\\

\begin{remark}{\rm
There is an alternative description of the connection between category
$\cO$ and the geometry of the Slodowy slice.  By work of the second
author in \cite{WebWO} and the localization theorem of Ginzburg
\cite{Gin08} and Dodd-Kremnitzer \cite[Theorem 7.4]{DK}, a singular
block $\cO_{n-k,k}$ of category $\cO$ with singularity type $(n-k,k)$
is equivalent to a subcategory of a category of sheaves for a
quantization of the structure sheaf on the Slodowy slice. These
sheaves are easily seen to be supported inside the closure of the set
of points attracted to the Springer fiber under a particular
$\C^*$-action.

The connection to our picture is related to the already mentioned
Koszul duality. Recall that the endomorphism algebra of a minimal
projective generator of the parabolic category $\cO^{n-k,k}$, turns
into the Ext-algebra of simples in the singular block $\cO_{n-k,k}$.
Thus, the two-fold appearance (in the context of the geometry of the
Springer) of the algebra studied in our paper is not
surprising. However, the precise connection between this
representation theoretic quantization construction and the
constructions of this paper has not been worked out yet.  }
\end{remark}

\subsection{An isomorphism of bimodules}
In \cite[Conjecture 5.9.2]{StrSpringer}, it was conjectured that for any two standard tableaux $S$ and $S'$, the cohomology  $H^\bullet(\com S\cap\com
{S'})$ is isomorphic, as a bimodule (with the above identifications), to the
Hom-space between the corresponding indecomposable projective modules over the
algebra $\ST^{n-k,k}$. We have the following more general result:

\begin{theorem}
\label{bimodules}
There are isomorphisms of graded algebras
\begin{equation}
\label{ringisos}
\Psi_x:\End_{\ST^{n-k,k}}(\tilde P(x\cdot0))\cong H^\bullet(\st {w(x)}),\quad x\in W^{n-k,k}
\end{equation}
such that under these identifications one can find isomorphism of graded bimodules
\begin{equation*}
\Psi_{x,y}:\Hom_{\ST^{n-k,k}}(\tilde P(x\cdot0),\tilde P(y\cdot0))\cong H^\bullet(\st {w(x)}\cap \st {w(y)})\langle d(x,y)\rangle
\end{equation*}
for any $x, y\in W^{n-k,k}$.
\end{theorem}

\begin{proof}
  Let $x\in W^{n-k,k}$. Consider the circle diagram
  $\overline{\operatorname{Cup}(x)}\operatorname{Cup}(x)$ and pick
  some odd vertex in each circle. If $I(x)$ denotes the set of these
  vertices, then $H^\bullet(\st {w(x)})\cong \C[\{x_i\}_{i\in
    I(x)}]/(\{x_i^2\}_{i\in I(x)})$. This follows from
  Theorem~\ref{cohomattract} by mapping $x_j$ for $j\in\Sm w$ to
  $a_jx_i$, where $i$ lies on the same circle as $j$ and $a_j=1$ if
  $j$ is odd, whereas $a_j=-1$ if $j$ is even. On the other hand
  $$\C[\{x_i\}_{i\in I(x)}]/(\{x_i^2\}_{i\in I(x)})\cong
  \dn^{c(x,x)}\cong \End_{\ST^{n-k,k}}(\tilde P(x\cdot0))$$ by mapping
  the $x_i$ to the $X$ associated with the circle where $i$ lies
  on. These isomorphisms define graded algebra isomorphisms $\Psi_x$
  of the form \eqref{ringisos}.  Similarly we define an isomorphism of
  vector spaces
\begin{multline*}
\Hom_{\ST^{n-k,k}}(\tilde P(x\cdot0),\tilde P(y\cdot0))\cong\dn^{\otimes c(x,y)}\\
\cong\C[\{x_i\}_{i\in I(x,y)}]/(\{x_i^2\}_{i\in I(x,y)})=H^\bullet(\st {w(x)}\cap \st {w(y)})\langle d(x,y)\rangle
\end{multline*}
by choosing a set $I(x,y)$ of odd vertices, one for each circle in the diagram
$\overline{\operatorname{Cup}(x)}\operatorname{Cup}(y)$. Hence we have
the family $\Psi_{x,y}$ of isomorphisms of vector spaces, which we
claim are isomorphisms of bimodules.

To see this let $1\otimes \ldots 1\otimes X\otimes 1\ldots 1$ be the
element in $\dn^{\otimes c(x,x)}$ where the $X$-factor corresponds to
a circle $C$ with leftmost vertex labeled by say $m$. It acts on
$\dn^{\otimes c(x,y)}$ by multiplication with $X$ on the factor
corresponding to the circle containing the vertex $m$. Under the
isomorphism $\Psi_{x,y}$ it corresponds to multiplication with
$a_rx_r$, where $r\in I(x,y)$ is on the same circle as $m$.

Under the isomorphism $\Psi_x$ the element $1\otimes \ldots 1\otimes
X\otimes 1\ldots 1$ is mapped to $a_ix_i$, where $i\in I(x)$ is the
chosen vertex on the circle $C$, hence acts by multiplication with
$a_ix_i$ on
$$H^\bullet(\st {w(x)}\cap \st {w(y)})\langle d(x,y)\rangle=\C[\{x_i\}_{i\in I(x,y)}]/(\{x_i^2\}_{i\in I(x,y)}).$$
Since all the elements in $I(x)$ and $I(x,y)$ are odd, this is the
same (thanks to the relations in $H^\bullet(\st {w(x)}\cap \st
{w(y)}$) as multiplication with $a_rx_r$ where $r\in I(x,y)$ lies on
the same circle as $i$.

Hence the isomorphisms $\Psi_{x,y}$ are equivariant with respect to
the left action. The arguments for the right action are completely
analogous.
\end{proof}

\section{Convolution algebras}
\label{sec:convolution-algebras}

\subsection{Definition of convolution}
\label{sec:def-conv}

The purpose of this section is to introduce an algebra structure on
the direct sum of all bimodules $H^\bullet(\st w\cap\st {w'})$ via a
convolution product and compare it with the algebra $\ST^{n-k,k}$.

Let $\taY$ be the disjoint union of the stable manifolds $\st w$ over
all weights $w$, and let $\tY$ be the disjoint union of the components
$\com S$ over all standard tableaux $S$, equipped with the obvious
maps $\taY\to \spf$ and $\tY\to \spf$, so
\begin{equation*}
\taY:=\bigsqcup_w\st w\rightarrow Y,\quad\quad\quad \tY:=\bigsqcup_S\com S\rightarrow Y.
\end{equation*}

Choose an element $f\in H^*(\taY\times_\spf\taY\times_\spf\taY)$, that is, a cohomology class of the intersection of each ordered triple of stable manifolds.

Both the cohomology groups
\begin{eqnarray}
\label{defK}
  H^\bullet(\taY\times_\spf\taY)&=&\bigoplus_{w,w'}H^\bullet(\st w\cap\st {w'})\\
  H^\bullet(\tY\times_\spf\tY)&=&\bigoplus_{S,S'}H^\bullet(\com
  S\cap\com {S'})
  \label{defH}
\end{eqnarray}
have a natural product structure defined
by a convolution product we denote $*_f$, given by pulling back, cupping together with $f$ and pushing forward on the diagram.
\begin{equation}
\label{product}
\xymatrix{
\taY\times_{Y}\taY\ar@{<-}[dr]^{p_{12}}&\\
&\taY\times_{Y}\taY\times_{Y}\taY\ar@{->}[r]^{\quad p_{13}}&\taY\times_{Y}\taY\\
\taY\times_{Y}\taY \ar@{<-}[ur]^{p_{23}}
}
\end{equation}

More explicitly, the product of two classes $\al\in H^\bullet(\st
w\cap\st {w'})$ and $\be\in H^\bullet(\st {w'}\cap \st {w''})$, is
$$\al*_f\be=(p_{13})_*(p_{12}^*\al\cup p_{23}^*\be\cup f)\in H^\bullet(\st w
\cap\st {w''}).$$  In essence, we do the usual convolution after
replacing the usual fundamental class on
$\taY\times_{Y}\taY\times_{Y}\taY$ with its cap product with $f$,
which one can think of as a ``virtual fundamental class.''  Since any
change of orientation can be absorbed into $f$, we will always give
these manifolds the complex orientation.

Obviously, if one is not very careful in one's choice of $f$, this
algebra will lack many desirable properties; in particular, it will
not be associative or graded.  However, for certain choices of $f$, it
will have these properties. In the next section we will describe a
good choice for $f$ which gives us our desired associative graded
algebra. In the section following that, we will describe the case
$f=1$, which gives an interesting non-associative graded algebra.

A toric analogue of our situation yielding to interesting associative
graded algebras is studied in \cite[\S 4]{BLPW}.  In that case, the
choice of $f$ was simply a careful choice of
orientations (so in that case, $f$ is degree 0, but not the identity).

Although our set up algebro-geometric, all the mentioned examples can
be interpreted as collections of Lagrangians in ambient symplectic manifolds and hence the following
appears natural:

\begin{question}{\rm
What conditions on $f$ must be satisfied to induce an associative product? Can this be expressed in terms of symplectic geometry as a general property of collections of Lagrangian submanifolds (as the components of the Springer fiber are inside the resolved Slodowy slice), maybe together with the data of a certain distinguished bundle on them?}
\end{question}

\begin{remark}{\rm
While superficially similar, the convolution constructions in this
paper are quite different in flavor from those of Chriss-Ginzburg
\cite{CG97}.  Our algebra is modeled on the behavior of coherent
sheaves (as we discuss later in this paper) or the Fukaya category of
a symplectic space in which the Springer fiber lies, not on those of
constructible sheaves.  If one took the constant constructible sheaves
on the components and looked at their Ext-algebra, one could also interpret this in terms of a convolution algebra on the
same underlying vector space, but with a different product structure and
{\it different grading}.

We wish to emphasize that these techniques from \cite{CG97} have no bearing on the structure of coherent sheaves, and thus could not be applied to prove Conjecture \ref{conj} below; it seems to be something of a remarkable coincidence that cohomology of these varieties describe Ext-algebras in different categories.  }
\end{remark}

\subsection{An isomorphism}

\begin{theorem}
\label{Genau}
There exists a class $f\in H^\bullet(\taY\times_\spf\taY\times_\spf\taY)$ such that the bimodule isomorphisms $\Psi_{x,y}$ from Theorem~\ref{bimodules} define in fact an isomorphism of algebras
  $$\ST^{n-k,k}\cong H^\bullet(\taY\times_\spf\taY;\C),$$
where the latter is given the multiplication $*_f$. \\

With the additional grading shifts as in \eqref{dshift} this isomorphism is compatible with the grading.  In the case $n=2k$, it induces an isomorphism of (graded) subalgebras $\KH^{k,k}\cong H^\bullet(\tY\times_\spf\tY)$.
\end{theorem}

\begin{proof}
  For purposes of the proof, it will be convenient to use cohomology classes $z_i=(-1)^ix_i$ as our generators, rather than $x_i$.

  Let $w',w,w''$ be row strict tableaux, with corresponding cup diagrams $C'=\m(w'),C=\m(w),C''=\m(w'')$.  The multiplication map $$\ST_{w',w}\otimes \ST_{w,w''}\to \ST_{w',w''}$$ is described by a cobordism from $\overline{C'}C\sqcup \overline{C}C''$ given by saddle moves on the pairs of cups in $C$ and connecting the ``loose ends'' of $C$, with the result given by $\overline{C'}C''$ (see \cite{BS1}).
We view this cobordism as a movie of length $n-k_w$ where we do one
saddle move or connection at a time (see \eqref{diagram} for an easy
example). Of course, this requires choosing a total order on the set
of cups and rays in $C$ compatible with the nesting partial order on cups.

Then we construct a sequence $Z_0,Z_1,\ldots, Z_{n-k_w}$ of varieties
$Z_i\subseteq \st{w'}\times \st{w''}$, one to each stage in the movie,
with equations corresponding to the state of the circle diagram at
that point in the cobordism.  By definition, $Z_i$ consists of all tuples $(F_\bullet, F_\bullet')$ of flags such that for any index $j\in \mathcal{E}(w)$
\begin{itemize}
\item $F_j=F_j'$ and $F_{\si(j)}=F'_{\si(j)}$ if $j$ lies on a cup and the corresponding saddle has been done already,
\item $F_{\si(j)}=N^{-\de(j)}(F_{j-1})$ and $F'_{\si(j)}=N^{-\de(j)}(F'_{j-1})$ if $j$ lies on a cup and the corresponding saddle has not been done already,
\item $F_j$ and $F_j'$ coincide with the space for the fixed point if instead $j$ lies on a line segment and $j$ hasn't been connected yet, and just
\item $F_j=F_j'$  if instead $j$ lies on a line segment and $j$ has been connected already.
\end{itemize}

Obviously, $Z_0\cong (\st{w'}\cap \st{w})\times(\st{w}\cap \st{w''})$ and $Z_{n-k_w}=(\st{w'}\cap \st{w})_\Delta,$ the diagonal embedding of that variety in $\st{w'}\times \st{w''}$.

 Furthermore, one can easily check that at the $i$th step of the cobordism, the dimension of $Z_i$ is the number of circles in the diagram, and that there are 7 possibilities (analogous to \cite[Section 6]{BS1}) for the next move:
\begin{itemize}
\item 2 circles become 1: the dimension drops by 1.
\item 1 circle becomes 2: the dimension jumps by 1.
\item a circle and line segment become a line segment: the dimension drops by 1.
\item a line segment ``births'' a circle: the dimension jumps by 1.
\item there is a saddle between two line segments: the variety is unchanged.
\item a line segment becomes a circle: the dimension jumps by 1.
\item 2 line segments become 1: the variety is unchanged.
\end{itemize}

In particular, for any $i$, we have either $Z_{i}\subseteq Z_{i+1}$
or $Z_i\supseteq Z_{i+1}$ as smooth subvarieties of $\st{w'}\times
\st{w''}$.  The most important claim of this proof is that at each
step of the sequence of varieties, pushing forward or pulling back under this
inclusion induces (the first author's reinterpretation of) Khovanov's
action of the cobordism up to that point on the cohomology of $Z_0$.
when we inductively identify the cohomology of $Z_i$ with the
corresponding vector space in Khovanov's construction (or rather the
extension of \cite{StrSpringer}) associated to the picture at the
$i$-th place in the movie.

Throughout, we will implicitly use the fact that $H^\bullet(Z_i)$ is a quotient of  $H^\bullet(\st{w'}\times \st{w''})$, and so use $z_{j,i}$ and $z_{j,i}'$ to denote the image of $z_j$ in $H^\bullet(Z_i)$ coming from the first and second factor respectively.  We then identify $H^\bullet(Z_i)$ with the vector space associated to the diagram of $Z_i$ by identifying the degree $2$ class on a circle with $z_{j,i}$ (resp. $z_{j,i}'$) if the circle passes through the
$j$th on the bottom (resp.\ top) row (see again \eqref{diagram} for the obvious definition of first and second row).  This class doesn't depend on which point on the circle we choose (see Section \ref{sec:cohom-pairw-inters}).

In each case, the pullback or push-forward map is a map of $H^\bullet(\st{w'}\times \st{w''})$-modules, so we only have to calculate the image of the identity of $H^*(Z_i)$ at each step of the cobordism.

The claim is true for $Z_0$ by definition. So assume it to be true for
$Z_{i-1}$ and consider the next move.

If the variety is unchanged, the asserted statement is trivial, since
the combinatorial multiplication also does nothing in this case
(\cite[\S 5.4]{StrSpringer} or \cite[Theorem 6.1 (i)]{BS1}).

Similarly when merging two circles,  the
pullback of 1 is sent to 1 and  $z_{j,i}$ and $z_{j,i}'$ are sent to
the same class (since the two corresponding line bundles
are isomorphic) which by our convention is $z_{j,i+1}=z'_{j,i+1}$.  This is precisely Khovanov's multiplication rule.

In the merging of a circle and line segment, the variable for the
circle ($z_{j,i}$ or $z_{j,i}'$) is sent to 0, since the line bundle become trivialized. The
same happens in the combinatorial multiplication, \cite[(6.1)]{BS1}.

Thus, the only tricky case is push-forward.  If we do a saddle move on
the cup from $j$ to $\si(j)$, then 1 is sent to
$x_j-x_{\si(j)}=(-1)^j(z_{j,i+1}+z_{\si(j),i+1})$, since $Z_i$ is the vanishing
set in $Z_{i+1}$ of    $\en^{\de(j)}\colon\quad\LB_{\si(j)}\to \LB_{j}$
and this is the first Chern class of $\Hom(\LB_{\si(j)},\LB_{j})$.
This is precisely the extension of Khovanov's rule (recall that $z_{j,i}$ is 0 if $j$ lies on a line segment, as in
\cite[(5.4.2)]{StrSpringer}), except for the difference of sign, which
we will deal with momentarily.

For closing a segment to a circle, we obtain $(-1)^{j+1}z_{j,i+1}$, which again
matches the rule from \cite[\S 5.4]{StrSpringer} (up to sign), since
this subvariety is given by the zero set of a section of
$\Hom(F_{j}/F_{j-1},\C)$.

Altogether, this proves that our geometric and the Khovanov-Stroppel algebraic multiplications agree up to signs (not necessarily an overall sign); however, we know exactly how the signs are off from the algebraic multiplication, and thus can correct for them by choosing a different orientation of $Z_i$.  We change the orientation by $(-1)^{\sum_{\ell=1}^i j_\ell}$ where
\begin{eqnarray*}
j_i=\begin{cases} j & \text{if the $i$th step is a saddle creating a circle on the $(i,\si(i))$ cup}\\
j+1 & \text{if the $i$th step closes a circle at $j$} \\
0& \text{otherwise.}
\end{cases}
\end{eqnarray*}
This precisely corrects for the signs which appeared in the description of the multiplication and we arrive at the Khovanov-Stroppel formulas.

In order to describe this in terms of cohomology classes on $\taY\times_\spf\taY\times_\spf\taY$, we need to use base change for clean intersections: if we have a diagram of cleanly intersecting submanifolds
\begin{displaymath}
\xymatrix@!=0.6pc{
 &X\ar@{<-_)}[dl]_{i_A}\ar@{<-^)}[dr]^{i_B} &\\
A\ar@{<-^)}[dr]_{j_A}&&B\ar@{<-_)}[dl]^{j_B}\\ &A \cap B&
}
\end{displaymath}

then by standard algebraic topology (analogous to \cite[Proposition 2.6.47]{CG97}), we have
that $$i_{B}^*(i_A)_*g=(j_A)_*(e(E)\cup j_{B}^*g)$$ where
$E=i_A^*j_A^*T_X/(j_A^*T_A+j_B^*T_B)$ is the excess bundle of the
intersection.

This shows that if we have any chain of manifolds $A_1\supseteq A_2 \subseteq A_3\supseteq\cdots \subseteq A_\ell$, such that the intersection of any subset of the $A_i$'s is clean,  we can always shorten the chain by doing base change, at the cost of multiplying by the Euler class of a bundle on one of the $A_h$'s.  By induction, the iterated push-pull can be described as a pulling back to $\bigcap_{i} A_i$, multiplying the Euler class of a bundle on $\bigcap_{i} A_i$, and pushing forward to $A_\ell$.

Applying this to the $Z_i$'s, we get a bundle on $\bigcap_i Z_i=\st{w'}\cap \st{w}\cap \st{w''}$ whose Euler class is the desired $f$ on that given component of $\taY\times_\spf\taY\times_\spf\taY$; the result follows.
\end{proof}

Since this proof describes the cohomology class $f$ in a somewhat implicit manner, let us attempt to give a more intuitive description at the cost of some loss of precision.  We can write Khovanov's cobordism in normal form, that means as union of 3 pieces: one where circles just join together, one where we add some handles to the cobordism, and one where the cobordism branches out to meet the circles of $\overline{C'}C''$.

Geometrically, each of these portions of the cobordism match up with parts of the convolution procedure:
\begin{itemize}
\item The merging portion of the cobordism corresponds to pull-back to the triple intersection $\st{w'}\cap\st{w}\cap\st{w''}$.
\item The handles portion corresponds to multiplying by $f$ on the triple intersection; in particular, the degree of $f$ is equal to twice the number of handles, i.e.\ the genus of the cobordism.
\item The branching portion corresponds to the push-forward to $\st{w'}\cap\st{w''}$.
\end{itemize}
In particular, if the cobordism has genus zero, than it only consists of merges and branches. In this case it follows from the proof of Theorem \ref{Genau} that the multiplication map
$$H^\bullet(\st w'\cap\st {w})\otimes H^\bullet(\st {w}\cap\st {w''})\rightarrow H^\bullet(\st {w'}\cap\st {w''})$$
giving rise to our desired algebra is just pulling back to the triple intersection and pushing forward, but with appropriate choices of orientations of the involved manifolds.

\subsection{Comparison with the natural choice of orientation}
For the sake of completeness we would like to indicate (without proof) in which sense the convolution algebra with our choice of orientation differs from the convolution algebra obtained when we choose the natural complex orientation. The difference will depend on a parameter $\alpha$, where we set $\alpha=1$ in case we chose the natural complex orientation, and $\alpha=-1$ for our choice of orientation.

\begin{theorem}
Let $w',w,w''$ be standard tableaux, with the corresponding cup diagrams $C'=\m(w'),C=\m(w),C''=\m(w'')$. The image of
$${_{w'}1_{w}}\otimes {_{w}1_{w''}}\in H^\bullet(\st {w'}\cap \st w)\otimes
H^\bullet(\st{w}\cap\st {w''})$$ in $H^\bullet(\st {w'}\cap\st {w''})$ under the convolution
product (in either case) can be calculated as follows: Place $\overline{C'}C$ over
$\overline{C}C''$ and consider the minimal cobordism $\Cob'$ from this
collection of circles to the collection of circles given by $\overline{C'}C''$ (see \cite{KhoJones}, \cite{BS1}).

If we consider this cobordism as a union of saddle moves corresponding to the set $S_\vee$ with respect to $w$ (with some fixed order compatible with the nesting) then ${_{w'}1_{w}}\otimes {_{w}1_{w''}}$ goes to the product $\prod_{i\in\Sm S} \varphi(i)$ where
\small
\begin{equation*}
\varphi(i)=\begin{cases} 1 & \text{if the saddle of $i$ joins two circles}\\
\alpha x_i + x_{\si(i)} & \text{if the saddle of $i$ creates two circles, and  $\gamma_{\si(i)}$ contains $\gamma_i$}\\
x_i +\alpha x_{\si(i)} & \text{if the saddle of $i$ creates two circles, and $\gamma_{i}$ contains $\gamma_{\si(i)}$}\\
\alpha x_i +\alpha x_{\si(i)} &\text{otherwise}
\end{cases}
\end{equation*}
\normalsize
where $\gamma_j$ denotes the created circle containing the vertex labeled by $j$ for any $j$.
\end{theorem}
\begin{proof}
omitted.
\end{proof}

\begin{ex}{\rm
If, for instance, $C'=\operatorname{Cup}(\nested)=C''$, $C=\operatorname{Cup}(\nxt)$
then we have the following possible sequence of diagrams describing
$\Cob'$ (which is in this case a pair of pants joining two circles to
one circle followed by a pair of pants which splits this one circle
into two)
\begin{eqnarray*}
\label{diagram}
\begin{picture}(200,80)(-50,-20)
\put(0,33){$\bullet$}\put(17,36){\oval(10,20)[t]}\put(7,33){\oval(10,20)[b]}
\put(10,33){$\bullet$}
\put(20,33){$\bullet$} \put(17,36){\oval(30,30)[t]}\put(27,33){\oval(10,20)[b]}
\put(30,33){$\bullet$}
\put(0,3){$\bullet$}\put(17,3){\oval(10,20)[b]}\put(7,6){\oval(10,20)[t]}
\put(10,3){$\bullet$}
\put(20,3){$\bullet$} \put(17,3){\oval(30,30)[b]}\put(27,6){\oval(10,20)[t]}
\put(30,3){$\bullet$}
\put(45,18){$\rightarrow$}
\put(60,33){$\bullet$}\put(77,36){\oval(10,20)[t]}\put(67,33){\oval(10,20)[b]}
\put(70,33){$\bullet$}
\put(80,33){$\bullet$} \put(77,36){\oval(30,30)[t]}\put(82,35){\line(0,-1){30}}
\put(90,33){$\bullet$}
\put(60,3){$\bullet$}\put(77,3){\oval(10,20)[b]}\put(67,6){\oval(10,20)[t]}\put(92,35){\line(0,-1){30}}
\put(70,3){$\bullet$}
\put(80,3){$\bullet$} \put(77,3){\oval(30,30)[b]}
\put(90,3){$\bullet$}
\put(105,18){$\rightarrow$}
\put(120,33){$\bullet$}\put(137,36){\oval(10,20)[t]}\put(122,35){\line(0,-1){30}}
\put(130,33){$\bullet$}
\put(140,33){$\bullet$} \put(137,36){\oval(30,30)[t]}\put(132,35){\line(0,-1){30}}
\put(150,33){$\bullet$}
\put(120,3){$\bullet$}\put(137,3){\oval(10,20)[b]}\put(142,35){\line(0,-1){30}}
\put(130,3){$\bullet$}
\put(140,3){$\bullet$} \put(137,3){\oval(30,30)[b]}\put(152,35){\line(0,-1){30}}
\put(150,3){$\bullet$}
\end{picture}
&\quad\quad\quad\quad&
\includegraphics{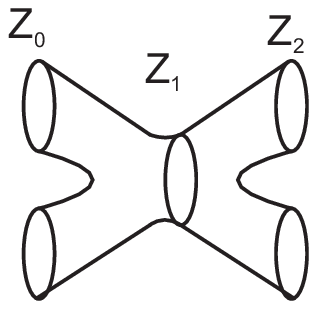}
\end{eqnarray*}

The element ${_{w'}1_{w}}\otimes {_{w}1_{w''}}$ is then mapped to $(x_1+\alpha x_2){_{w'}1_{w''}}$, since the only place where a circle is split into two is at the cup/cap pair attached to the vertices $1$ and $2$ (from the left). Alternatively we could have chosen the sequence where we first remove the cup/cap pair attached to the vertices $1$ and $2$, so that ${_{w'}1_{w}}\otimes {_{w}1_{w''}}$ is then mapped to $(x_3+\alpha x_4){_{w'}1_{w''}}$ which equals $(x_1+\alpha x_2){_{w'}1_{w''}}$ in $H^\bullet(\st {w'}\cap\st {w''})$. The result will always be independent from the chosen sequence, since any such sequence describes the convolution product. If we swap the roles of $C'$ and $C''$ then ${_{w'}1_{w}}\otimes {_{w}1_{w''}}$ would be mapped to $(\alpha x_1+\alpha x_3){_{w'}1_{w''}}$ in $H^\bullet(\st {w'}\cap\st {w''})$.
}
\end{ex}

If $\alpha=-1$, then the resulting algebra is not associative, if $\alpha=1$ then this is exactly Khovanov's arc algebra (with the extension from \cite{StrSpringer}). It seems natural to search for a topological construction making transparent the difference between these two
algebra structures on the same vector space. Our suggestion is to use
a TQFT-like procedure like Khovanov's, but one which is sensitive to
the embedding of cobordisms in 3-space. This is what we propose to call an {\it embedded
  2-dimensional TQFT}.

Equivalently, one can say that our cobordisms keep track of the
nestedness of the circles. In particular, there will be two types of
pair of pants cobordisms, namely one which connects one circle with
two disjoint, not nested circles in the usual embedding for trousers
and a second ``unusual'' one which connects one circle with two
disjoint, but nested circles, with one of the trouser legs pushed down
the middle of the other, see Figure \ref{fig:pants}.

\begin{figure}[htb]
\includegraphics{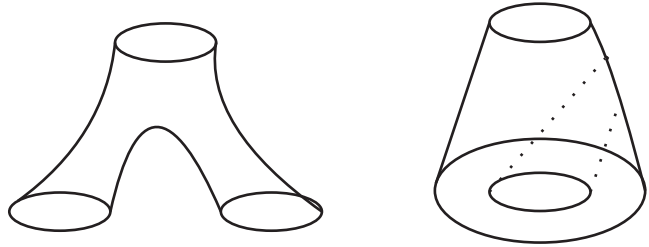}
\caption{Embedded pair of pants: The usual one for not nested circles, the unusual one for nested circles.}
\label{fig:pants}
\end{figure}

For instance, the minimal cobordism displayed in the previous example
would be a composition of a usual pair of pants connecting two circles
to one followed by a generalized pair of pants splitting one circle
into two nested circles. We now define an embedded version of
Khovanov's algebra by assigning the following maps to the pair of
pants morphisms:
\begin{itemize}
\item To a usual pair of pants joining two (not-nested) circles to one
  circle, we associate the multiplication $m:\dn\otimes\dn\rightarrow
  \dn$, $1\otimes 1\mapsto 1$, $X\otimes 1\mapsto X$ $1\otimes
  X\mapsto X$, $X\otimes X\mapsto 0$.
\item To the reverse cobordism, splitting one circle into two (not-nested) circles, we associate the comultiplication
  $\Delta:\dn\rightarrow \dn\otimes\dn$, $1\mapsto -X\otimes
  1-1\otimes X$, $X\mapsto-X\otimes X$.\\
(So far it is exactly the setup of \cite{KhoJones}, except that our $-X$ is $X$ there.)
\item To the ``unusual'' pair of pants joining two nested circles to
  one circle, we associate the map $m':\dn\otimes\dn\rightarrow \dn$,
  $1\otimes 1\mapsto 1$, $X\otimes 1\mapsto X$, $1\otimes X\mapsto
  -X$, $X\otimes X\mapsto 0$, where the first tensor factor is
  associated with the outer circle and the second with the inner
  circle.
\item To the reverse cobordism, we associate the linear map
  $\Delta':\dn\rightarrow \dn\otimes\dn$, $1\mapsto X\otimes
  1-1\otimes X$, $X\mapsto-X\otimes X$, where again the first tensor
  factor is outer and the second is inner.
\end{itemize}
Keeping track of the nestedness using
the rules above describes exactly the (non-associative) multiplication on the convolution algebra with the ordinary complex orientation.

\section{Coherent sheaves and cup functors}
\label{sec:coher-sheav-comp}

In this section, we want to connect our approach with the one of
\cite{CK}, where an alternative (geometric) categorification of the
Jones polynomial was obtained. It agrees on the $K_0$-group level with
the Reshetikhin-Turaev tangle invariant \cite{RT} associated with
$\cU_q(\mathfrak{sl}_2)$, hence also with the decategorification of
\cite{StDuke} which in turn restricts to Khovanov's functorial
invariant. The precise categorical or functorial connection between
the geometric and algebraic-representation theoretic picture is
however open at the moment. In the following, we give some partial
results which indicate that the geometric picture might differ
slightly from the algebraic one.  We note that our results partially
overlap with those obtained independently in \cite{Anno}.

\subsection{Geometric background}
\label{sec:background}

Now, we consider the Springer fiber as a Lagrangian subvariety inside
a larger smooth space.  This ambient space is best defined as the
pre-image under the Springer resolution of a normal slice to the
nilpotent orbit through $N$ at $N$.  We denote this space by $\Snk$.
Our Springer fiber is included as the fiber over $N$. The interested
reader can consult \cite{MV08} for details. For our purposes, the only
important fact about the varieties is that they are smooth, and each
component of the Springer fiber is a Lagrangian subvariety inside
$\Snk$.  These spaces were for instance used in the geometric
construction of knot invariants via Floer homology in the work of
\cite{SS} and \cite{Manolescu}.

In the case where $ n=2k $, this variety has a more convenient
description, which played an important role in the work of Cautis and
Kamnitzer \cite{CK}, who used a compactification of it to define
homological knot invariants. So from now on let $n=2k$. Let $M$ be the
nilpotent endomorphism of $\mC^{2n}$ with two equally sized Jordan
blocks.  Let $\{p_1, p_2,\ldots p_n,q_1,q_2,\ldots q_n\}$ be the basis
of $\mC^{2n}$ such that $M$ has Jordan Normal Form (with
$Mp_i=p_{i-1}$ and $Mq_i=q_{i-1}$) with the $\mC^*$-action as before
on $V$. Now define the space of flags
\begin{equation*}
Z_{n}=\{F_0\subset F_i\subset \ldots F_{n-1}\subset F_n\subset \mC^{2n}\mid \operatorname{dim}_\mC F_i=i, M F_i\subset F_{i-1}\}.
\end{equation*}

We can identify our original vector space $V$ with the span of the
$p_i$, $q_i$ for $1\leq i\leq n$, with the endomorphism $M$
restricting to the nilpotent endomorphism $N$.  Thus, we can identify
$Y$ with the subset of $Z_{n}$ where $F_n=V$.  Furthermore,
$\EuScript{S}_{k,k}$ can be identified with the subset of $Z_n$ where
the projection of $F_n$ onto $V$ (by forgetting the coordinates with
higher indices) is an isomorphism.

In \cite[Section 4]{CK}, the authors define functors between the bounded derived categories $\mathcal{D}(Z_{n'})$ of ($\mC^*$-equivariant) coherent sheaves on $Z_{n'}$ (for varying $n'$) which provide a categorified tangle/knot invariant, in the following sense: to each $(n_1,n_2)$-tangle, there is an associated functor from $\mathcal{D}(Z_{n_1})$ to $\mathcal{D}(Z_{n_2})$ which is a tangle invariant, up to isomorphism, and decategorifies to the Reshetikhin-Turaev tangle invariant associated with (the quantum group of) $\mathfrak{sl}_2$ on the level of the $K_0$-group.

In fact, for all $k$, the space $\Snk$ is embedded in $Z_n$ matching the obvious inclusion of the Springer fiber (see \cite{MV08}).  The compactification obtained by closing this embedding seems to be a likely candidate for extending \cite{CK} beyond the case of blocks of equal size.  However, we will not pursue this idea further in this paper.

Let $\Coh(Z_n)$ be the category of coherent sheaves on $Z_n$ with its bounded derived category $\mathcal{D}^b(Z_n)$.

For our purposes, the $\mC^*$-action carefully tracked in \cite{CK} is unnecessary, so we will ignore it.   Since all the functors of concern are defined by Fourier-Mukai transforms,  they have non-equivariant analogues.

Note that, $Z_{0}$ is just a point, and so $\Coh(Z_{0})$ is the category of vector spaces over $\mC$.

If $C$ is a cup diagram corresponding to a standard tableau $S$ with two rows of size $k$, we can view it as a $(0,2k)$-tangle and consider the associated functor
$$\vp_C:\quad\mathcal{D}^b(Z_0)=\mathcal{D}^b(\Vect) \to\mathcal{D}^b(Z_n)$$ as defined in \cite{CK} (the interested reader may note Equation (\ref{eq:1}) below serves as an inductive definition of this functor).  In general, the functors associated with crossingless tangles are not exact in the standard $t$-structure on $\Coh(Z_n)$ (though of course, they are exact in the triangulated sense). In the special case of a $(0,2k)$-tangle, the situation is much easier: First of all, the functor maps a vector space to an actual sheaf (i.e. is exact in the usual $t$-structure), hence defines (or comes from) a functor
\begin{eqnarray}
\label{cupdiagfunctor}
\vp_C:\quad\Coh(Z_0)=\Vect \to\Coh(Z_n).
\end{eqnarray}
Secondly, as with any exact functor from vector spaces to any abelian category, $\vp_C$ is already determined by its value on $\mC$.

\subsection{Half-densities}
\label{sec:half-densities}

We let $\hd(\com S)$ denote a square-root of the canonical bundle on the component $\com S$.  This sheaf exists by the theorem below (but more generally, it exists at least as a twisted sheaf) and is unique, since the Picard group of any iterated $\CP^1$-bundle is torsion-free.

\begin{lemma}
Each component $\com S$ carries a unique square-root of the canonical bundle.
In fact, $\hd(\com S)\cong\bigotimes_{i\in \Sm S} \LB_i$.
\end{lemma}
\begin{proof}
Abbreviate $A=\com S$. As in any bundle, one can always compute the canonical bundle on the total space as the product of the canonical bundle on the base and the relative canonical bundle.  Since each component is fibered over one for a smaller diagram, to show the result by induction, we need only show that the relative canonical bundle of that fibration has a square root.

Let $i$ be an index such that $\si(i)=i+1$.  In this case, our fibration is
$$\mathbf{q}_i:A\to A',$$ where $A'$ is the component for our cup diagram with the cup from $i$ to $i+1$ deleted.  Since $\LB_i^{-1}$ is isomorphic to $\cO(1)$ on the fibers, we have that our fibration is the projectivization of the bundle $\mathbf{q}_{i*}\LB_i^{-1}\cong \LB_j\oplus\LB_j^{-1}$ where $j$ is the left end of the cup immediately nested over $i$.  Thus, we have an exact sequence \begin{equation*}
 0\to \Omega_{A/A'}\to \Hom(\mathbf{q}_i^*(\LB_j\oplus\LB_j^{-1}), \LB_i)\to\Hom(\LB_i,\LB_i)\to 0
\end{equation*}
The multiplicativity of determinants in exact sequences shows that \begin{equation*}
   \Omega_{A/A'}\cong\det( \Omega_{A/A'})\cong \det\left(\Hom(\mathbf{q}_i^*(\LB_j\oplus\LB_j^{-1}), \LB_i)\right)\cong \LB_i^{2}
\end{equation*}
Thus, $\Omega_{A/A'}^{1/2}\cong \LB_i$.
On the other hand, $\mathbf{q}_i^*\Omega_{A'}\cong\bigotimes_{j\in\Sm S\setminus \{i\}} \LB_i$.  Thus, the result follows by induction.
\end{proof}

In fact, these square roots are exactly the images of the $1$-dimensional vector space under the functors $\varphi_C$ associated to cup diagrams:

\begin{theorem}
\label{funcCK}
Let $W$ be any finite dimensional vector space. Then
  $$\vp_C(W)\cong W\otimes_\C \hd(\com S).$$
\end{theorem}
\begin{proof}
  Our proof is by induction.  Assume that the result is true for all smaller $n$, in particular for the corresponding cup diagrams with less than $n$ points. This set of diagram include for instance the diagram $C'$ which is $C$ with one of its minimal cups removed.  Denote by $S'$ the corresponding standard tableau and let $j$ and $j+1$ be the endpoints of this cup.

Then if ${\bf i}={\bf i}_j$ is the inclusion of the locus where $N(\Fl_{j+1})=\Fl_{j-1}$ holds, and ${\bf q}={\bf q}^j$ is the projection defined on this locus to $Z_{n-2}$ given by forgetting $\Fl_j$ and $\Fl_{j+1}$ as well as applying $N$ to all subspaces larger than $\Fl_{j+1}$, we have (\cite[4.2.1]{CK}) the equation
  \begin{equation}
    \vp_C(W)={\bf i}_*(\LB_j\otimes {\bf q}^*(\vp_{C'}(W)).
  \end{equation}
By induction, our proposition holds for $C'$, so this equation becomes
\begin{equation*}
  \vp_C(W)={\bf i}_*(\LB_j\otimes {\bf q}^*(W\otimes_\C\hd(\com {S'})). \label{eq:1}
\end{equation*}

On the other hand, we have the usual exact sequence of normal bundles
\begin{equation*}
0\to {\bf q}^*\mathcal{N}_{\com{S'}/Y_{n-2}}\to \mathcal{N}_{\com{S}/Y_n}  \to \left.\LB_{j+1}^*\otimes\LB_j\right|_{\com{S}}\to 0.
\end{equation*}
Since $\LB_{j+1}^*|_{X_n^j}\cong\LB_j|_{X_n^j}$, we see that $\canb(\com{S})\cong {\bf q}^*\canb(\com{S'})\otimes \LB_j^{\otimes 2}$, so $\hd(\com{S})\cong {\bf q}^*\hd(\com{S'})\otimes \LB_j$.  Applying this in equation (\ref{eq:1}), we obtain the desired result.
\end{proof}

On the way of trying to connect the different categorifications of the Turaev-Reshetikhin tangle invariants one could hope for an isomorphism of rings
\begin{equation*}
\Ext^\bullet_{\Coh(\Snk)}(i_*\hd(A),i_*\hd(A)) \cong \End(\tilde{P}(x\cdot0))
\end{equation*}
where $P(x\cdot0)$ is the indecomposable projective module associated with a component $A$ under the isomorphisms of \eqref{ringisos}, or more generally a formula like
\begin{equation}
\label{CohP}
\Ext^\bullet_{\Coh(\Snk)}(i_*\hd(A),i_*\hd(B)) \cong \End(\tilde{P}(x\cdot0),P(y\cdot0)).
\end{equation}
as graded vector spaces (up to our usual shifts).
On the other hand, based on work, such that of Leung (\cite{Leu02}), one might expect that
\begin{equation}
\label{Ends}
\Ext^\bullet_{\Coh(\Snk)}(i_*\cO_A,i_*\cO_A) \cong H^*(A)
\end{equation}
or more generally
\begin{equation}
\label{wrong}
\Ext^\bullet_{\Coh(\Snk)}(i_*\cO_A,i_*\cO_B) \cong H^\bullet(A\cap B)
\end{equation}
as graded vector spaces (up to our usual shifts), where $\cO_A$ denotes the structure sheaf on $A$.
In the following we will show that, in fact, all of the above statements are true, except the last one (which might appear as a surprise).

The importance of these square roots of canonical bundles (the so-called {\bf half-densities}) in connection with derived categories of coherent sheaves and the failure of \eqref{wrong} have previously been noticed by physicists in connection with the so-called {\bf Freed-Witten anomaly}, (see \cite{FW99}).

A mathematical manifestation of this phenomenon appears when considering the spectral sequences computing the $\Ext^\bullet$-groups of the square roots of the canonical sheaves in contrast to the ones computing the $\Ext^\bullet$-groups of the structure sheaves of these varieties, as carefully explained for instance in papers such as \cite{KS02}, \cite{Sha04}.

The crucial point hereby is that by the adjunction formula for the canonical bundle on a subvariety (\cite[Proposition 2.2.17]{Huy}), using half-densities instead of structure sheaves compensates for the appearance of the normal bundle in the $E_2$-term of the spectral sequence of \cite{KS02} which we use below.\\

Let now $n=(n-k)+k$ as usual. Let $A,B$ be components in the corresponding Springer fiber $Y$ included in the resolution to the Slodowy slice $\Snk$. Let $i:A\hookrightarrow \Snk$, and $j:B\hookrightarrow \Snk$ be the natural inclusions. The formula~\eqref{CohP} is by Theorem~\ref{bimodules} equivalent to the following result:

\begin{theorem}
\label{halfdensities}
There is an isomorphism of graded vector spaces
\begin{equation*}
\Ext^\bullet_{\Coh(\Snk)}(i_*\Omega(A)^{1/2},j_*\Omega(B)^{1/2})\cong  H^\bullet(A\cap B)\langle d(A,B)\rangle,
\end{equation*}
\end{theorem}
\begin{proof}
First, note that since $\Snk$ is holomorphic symplectic and the components of the Springer fiber are Lagrangian, so the symplectic form induces an isomorphism between the normal bundle and cotangent bundle.  Further, this shows that on an intersection, the quotient $$T_{\Snk}|_{A\cap B}/(T_A|_{A\cap B} + T_B|_{A\cap B})$$ will be the cotangent bundle $T^*_{A\cap B}$.

  Given these facts, the result follows almost immediately from  \cite[Theorem A.1]{CKS03} (though the theorem appeared with a less complete proof in \cite{KS02}).   In our case, this gives a spectral sequence
  \begin{equation*}
    H^p(A\cap B, \wedge^{q}T^*_{A\cap B})\cong H^{p,q}(A\cap B) \Rightarrow \Ext_{\Coh(\Snk)}^{p+q+d(A,B)}(i_*\Omega(A)^{1/2},j_*\Omega(B)^{1/2})
  \end{equation*}
  where $H^{p,q}$ denotes the usual Dolbeault cohomology.  The   first Chern classes of line bundles (which lie in   $H^{1,1}(A\cap B)$) generate $H^{p,q}(A\cap B)$, so it has only $(p,p)$ Dolbeault cohomology.  Thus, this spectral   sequence has no non-trivial differentials, and we obtain the   desired isomorphism.
\end{proof}

\begin{corollary}
There is an isomorphism of graded vector spaces
\begin{equation*}
\Ext^\bullet_{\Coh(\Snk)}\Big(\bigoplus_{A}i_{A*}\Omega(A)^{1/2},\bigoplus_{A}i_{A*}\Omega(A)^{1/2}\Big)\cong  H(\tY\times_\spf\tY),
\end{equation*}
where the sum runs over all irreducible components $A$.
\end{corollary}


Of course, both the left and right side of this isomorphism have natural ring structures given by Yoneda product and by convolution.
The statement of the following conjecture
would give a very explicit description of the Ext-algebra of half-densities:
\begin{conjecture}
\label{conj}
There is an isomorphism of algebras
\begin{equation*}
\Ext^\bullet_{\Coh(\Snk)}\Big(\bigoplus_{A}i_{A*}\Omega(A)^{1/2},\bigoplus_{A}i_{A*}\Omega(A)^{1/2}\Big)\cong  H(\tY\times_\spf\tY).
\end{equation*}
\end{conjecture}

\begin{remark}
{\rm
Of course, this $\Ext$-algebra is, as a vector space, {\em also} isomorphic to Khovanov's arc algebra, and at the moment, the authors are unsure as to which product on this vector space corresponds to Yoneda's. Having clarified this conjecture it wouldn't be too difficult to extend it to \eqref{defK}.
}
\end{remark}

An affirmative or negative answer to this conjecture would direct us toward further questions on the correct geometric perspective on knot homology:

\begin{question}
{\rm
Is it possible to construct a functorial tangle invariant and categorification of the Jones polynomial using our new convolution algebras? If so what is the relation to previous geometrical ones (\cite{CK}, \cite{SS}, \cite{Manolescu}) and to algebraic/representation theoretic approaches (\cite{Khotangles}, \cite{StDuke})?
}
\end{question}

As was noted in
\cite{Anno}, these half-densities are so-called {\bf exotic sheaves} as introduced by Bez\-rukav\-ni\-kov \cite{BezNon}. This suggests that the conjecture and questions above could be investigated using the noncommutative Springer resolution and related techniques of algebraic geometry.\\

We can perform a partial verification of Conjecture~\ref{conj}, considering only a single component at a time.

 \begin{theorem}
\label{Extofcomp}
Let $A$ be an irreducible component of $Y$ and $i:A\hookrightarrow\Snk$ the inclusion. Let $\cO_A$ be the structure sheaf on $A$.
Then there are isomorphisms of graded rings
\begin{equation*}
\Ext^\bullet_{\Coh(\Snk)}(i_*\hd(A),i_*\hd(A))\cong \Ext^\bullet_{\Coh(\Snk)}(i_*\cO_A,i_*\cO_A)\cong H^\bullet(A).
\end{equation*}
\end{theorem}

\begin{remark}
{\rm
Note that thanks to \eqref{ringisos} the rings appearing in the theorem can also also be identified with the endomorphism rings of    indecomposable projective and at the same time injective modules in the associated parabolic category $\cO$ for $\mathfrak{sl}_n$.
Based on the results of this paper, the slight generalization from components to arbitrary stable manifolds shouldn't be too difficult by mimicking \cite[Section 6]{BS2}.
}
\end{remark}

\begin{proof}[Proof of Theorem~\ref{Extofcomp}]
The first isomorphism follows from the fact that $\hd(A)$ deforms to a global line bundle on $\Snk$, the pullback of $\prod_{i\in S}\LB_i$ from $Z_n$.  (It's worth noting, this isomorphism does {\em not} hold in general.)

To compute the Ext-algebra on the left hand side we first compute the Ext-sheaves $\sExt^\bullet(i_*\cO_A,i_*\cO_A)$.
The irreducible component $A$ is smooth, hence a local complete intersection (\cite[Example 8.22.1]{Hartshorne}). Since we can work locally, we might assume that $A$ is the zero locus of a regular section $s\in H^0(E)$ for some bundle $E$ on $Z$.

Then we have the Koszul resolution
\begin{equation}
\label{Koszul}
0\rightarrow \bigwedgie n E^*\rightarrow \bigwedgie {n-1} E^*\rightarrow \ldots \bigwedgie 1 E^*\rightarrow E^*\rightarrow\cO_Z\rightarrow i_*\cO_C\to 0.
\end{equation}
where the differential maps $f_1\wedge f_2\wedge\ldots \wedge f_r\in \bigwedgie r E^*$ to
$$\sum_{i=1^r}(-1)^{i-1}f_i(s) f_1\wedge f_2\wedge\ldots \wedge f_{i-1}\wedge f_{i+1}\wedge \ldots\wedge f_r.$$ The Koszul complex is exact, since $s$ is a regular section (\cite[page 688]{GH}).

The beginning of the resolution \eqref{Koszul} defines a surjection
\begin{equation}
\label{surjE}
E^*\rightarrow\cI\rightarrow 0
\end{equation}
where $\cI$ is the ideal sheaf of $A$ in $Z$. Tensoring with $i_*\cO_A$, we get a surjection $i_*E^*\rightarrow\cI/\cI^2=\mathcal{N}^*_{A/Z}$. This map is an isomorphism for dimension reasons.

Now $\sExt^\bullet(i_*\cO_A,i_*\cO_A)$ can be calculated as the cohomology sheaves of the complex $i_*\wedge^\bullet E$. Since $i_*s=0$, the differentials in this complex are all zero, hence
\begin{equation}
\label{isoExtN}
\sExt^\bullet(i_*\cO_A,i_*\cO_A)\cong\wedge^*\mathcal{N}_{A/Z}
\end{equation}
as graded vector spaces.

We have to compare the ring structure. We first claim that there is a map of differential graded algebras
\begin{equation*}
c:\quad \bigwedgie \bullet E \rightarrow\sExt^\bullet(i_*\cO_A,i_*\cO_A)
\end{equation*}
sending $\xi\in\wedge^r E$ to the contraction with $\xi$, denoted  $c_\xi$. The differentials in the Koszul complex \eqref{Koszul} are given by contraction $c_s$ with the section $s$, and $c_\xi$ and $c_s$ super commute. Therefore $c(\xi)$ is a chain map of degree $k$. Since contraction satisfies $c_\xi\circ c_\zeta=c_{\xi\wedge\zeta}$, the map $c$ intertwines the wedge product on the source space with the composition in the target space. Passing to cohomology, we obtain that \eqref{isoExtN} is an isomorphism of algebras.

Since the component $A$ is Lagrangian inside $Z$, we have a canonical isomorphism between the normal bundle of $A$ in $Z$ and the cotangent bundle of $A$, in formulas $\mathcal{N}_{A/Z}\cong T^*_A$.

It thus follows from the isomorphism of (\ref{isoExtN}) that the
cohomology of the Ext-sheaf $\sExt^\bullet(i_*\cO_A,i_*\cO_A)$ is canonically isomorphic to the Dolbeault
cohomology $H^\bullet(A; \wedge^\bullet T^*_A)$ (here we abuse
notation, and identify the vector bundle $\wedge^q T^*_A$ with its sheaf of
holomorphic sections) with its usual product induced by $\wedge$. By
the Hodge theorem, this is isomorphic to the de Rham cohomology
$H^\bullet(A;\mC)$ equipped with the cup product.

Thus, we have the local-global spectral sequence
\begin{multline*}
 E_2^{p,q}: H^p(A;\wedge^q T^*A)=H^p(A;\wedge^q\mathcal{N}_{A/Z})=H^{p+q}(A;\sExt^\bullet(i_*\cO_A,i_*\cO_A))\\
 \Longrightarrow \Ext^{p+q}_{\Coh(\Snk)}(i_*\cO_A,i_*\cO_A).
\end{multline*}
This sequence collapses due to the Hodge diamond only having diagonal support, as in the proof of Theorem \ref{halfdensities}, and thus induces a ring isomorphism from $H^\bullet(A;\mC)$ to the ring $\Ext^\bullet_{\Coh(\Snk)}(i_*\cO_A,i_*\cO_A)$.
\end{proof}

\section{Exotic sheaves and highest weight categories}

In fact, we would like to propose a correspondence between weight sequences and certain sheaves on $Z_n$, which extends that sending a full crossingless matching on $n$ points to half-densities on the corresponding component of the Springer fiber.\\

Let $w$ be a weight sequence of length $n$. We denote by $r(w)$ be the number of cups in $C(w)$.  Let
\begin{equation*}
Z_w=\{\Fl_*\in  Z_n|\Fl_{i-1}=\en^{\de(i)}\Fl_{\si(i)} \text{ for $i$ and $\si(i)$
    connected in $C(w)$}\}.
\end{equation*}
with its embedding $j=j_w:Z_w\rightarrow Z_n$. If $r(w)=1$ we have the
map $q:Z_w\rightarrow Z_{n-2}$ as in \eqref{cupdiagfunctor}, and in
general a map $p: Z_w\rightarrow Z_{n-2r(w)}$ by taking compositions
of such maps, one for each cup.

Consider the line bundle $\LB_w=\bigotimes_{i\in \Sm w}\LB_i$ on $Z_w$
and set $\cM_w=j_*\LB_w$. In the setup of \cite{CK}, the latter has the
following description: to the cup diagram $C(w)$, Cautis and Kamnitzer
associated a functor $F: D^b(Z_{n-2r(w)})\rightarrow D^b(Z_n)$ and (by
comparing the definitions) we have $j_*\LB_w= F(\LB_{\tilde{w}})$,
where $\tilde{w}$ is the induced weight sequence on the orphaned
points of $C(w)$.

We have the following two extreme cases:
\begin{itemize}
\item If $r(w)=0$, then $F$ is just the identity functor and we have
  $\cM_w=\LB_w$.
\item If $r(w)=k$, then $Z_w$ is just a point and in fact,
  $\cM_w=\vp_{C(w)}(\mC)$ as in Theorem~\ref{funcCK}.
\end{itemize}

Let $\Theta_w$ be the set of weight diagrams which differ from $w$ by
switching the signs on opposite ends of any number of cups in
$\op{Cup}(w)$. For an object $M\in D^b(Z_n)$ we denote by $[M]$ its
class in $K_0(D^b(Z_n))$. Then the following holds

\begin{proposition}
\label{multiplicities}
  \begin{equation}
   [\cM_w]=\sum_{w'\in\Theta_w}(-1)^{\ell(w)-\ell(w')}[\LB_{w'}]
  \end{equation}
 In particular, the classes of $[\cM_w]$ and $[\LB_w]$ span the same sublattice of the Grothendieck group.
\end{proposition}

\begin{proof}
 We induct on the number of cups in $C(w)$.  If this is 0, we have reduced to the fact that $\cM_w\cong \LB_w$ in this case.  Otherwise, we can write $\cM_w=\vp_i(\cM_{v})$ where $i$ is on the left end of a minimal cup in $C(w)$ and $v$ is the induced weight sequence on $S-\{i,i+1\}$.  Now, we can assume that $[\cM_v]=\sum_{v'\in\Theta_v}(-1)^{\ell(v)-\ell(v')}[\LB_{v'}]$.
Let $v^+$ be $v$ with the cup at $i,i+1$ added and marked with $\vee\wedge$, and $v^-$ be the same, but with $\wedge\vee$ at $i,i+1$ instead.  Then, as we noted previously, we have an exact sequence
\begin{equation}\label{cup-exact}
0\to\LB_{v^-}\to\LB_{v^+}\to\vp_i(\LB_v)\to 0
\end{equation}
and thus in the Grothendieck group, $[\vp_i(\LB_v)]=[\LB_{v^+}]-[\LB_{v^-}]$.

Note that $\Theta_w=\Theta_{v}^+\sqcup \Theta_v^-$, and $\ell(v)\equiv \ell(v^+)\equiv \ell(v^-)+1\pmod 2$ so
\begin{eqnarray*}
 [\cM_w]=[\vp_i(\cM_v)]&=&\sum_{v'\in \Theta_v}(-1)^{\ell(v)-\ell(v')}\left([\LB_{(v')^+}]-[\LB_{(v')^-}]\right)\\
 &=&
  \sum_{w'\in\Theta_w}(-1)^{\ell(w)-\ell(w')}[\LB_{w'}]
\end{eqnarray*}
\end{proof}

\begin{remark}
\label{KL}
{\rm Proposition~\ref{multiplicities} should be compared with \cite[(5.12)]{BS1} which implies that  $[\cM_w]=\sum_{w'\in\Theta_w}d_{w,w'}(-1)[\LB_{w'}]$, where $d_{w,w'}$ is a Kazhdan-Lusztig polynomial (arising from perverse sheaves on Grassmannians).
}
\end{remark}

By the Cellular Fibration Lemma \cite[Lemma 5.5.1]{CG97} and \cite[Theorem 6.2]{CK}, the $\LB_w$'s generate $D^b(Z_n)$, and in fact are a basis of the Grothendieck group. As a consequence of Theorem~\ref{multiplicities} we have the following:
\begin{corollary}
 The objects $\cM_w$ generate the category $D^b(Z_n)$ and are a basis of its Grothendieck group.
\end{corollary}
By Remark~\ref{KL}, the transformation matrix between the two bases is given by Kazhdan-Lusztig polynomials.\\

Following ideas of Bezrukavnikov, we now define a $t$-structure on
$D^b(Z_n)$ (not the standard one) for which the $\cM_w$ form a
complete set of simple objects in the heart. This heart will then be
equivalent to the category of finite dimensional modules over our
convolution algebra $K_n$. The algebra $K_n$ is quasi-hereditary with
the standard modules given by the line bundles $\LB_w$'s.

First, we define the necessary ordering on the set of weights.  This is the standard ordering on weights which can be explicitly given in this case by saying that $w\leq v$ if for each $i$, there are more $\vee$'s in the last $i$ indices for $v$ than $w$. Alternatively, it's the partial ordering generated by the basic relation that changing $\up\down$ to $\down\up$ is getting smaller in the ordering.

\begin{lemma}
\label{semisimple}
The full additive category generated by the $\cM_w$'s is semisimple.
Let still be $n=2k$. Let $w$, $w'$ weights. Then $\operatorname{Hom}_{D^b(\Coh(Z_n)}(\cM_w, \cM_{w'})=\{0\}$ for $w\not=w'$ and $\operatorname{Hom}_{D^b(\Coh(Z_n)})(\cM_w, \cM_{w'})=\mC$ otherwise.
\end{lemma}

\begin{proof}
Claim: let $d(w,w')$ be as in \eqref{dshift} then either $\Ext^i_{\Coh(Z_n)}(\LB_w,\LB_w)$ is trivial or its minimal nonzero degree is $i=d(w,w')$, and so the lemma follows directly. Note that in case $w$, $w'$ correspond to standard tableaux, then the claim is clear by Theorem~\ref{halfdensities}. It of course also holds for $n=2$.

Assume first $w$ is minimal in the partial order $\leq$ and $w'$ is arbitrary. If $w=w'$, then $\Ext^i_{\Coh(Z_n)}(\cM_w,\cM_w)\cong\Ext^i_{\Coh(Z_n)}(\LB_w,\LB_w)\cong H^i(\LB_w^*\otimes \LB_w)\cong H^i(\mathcal{O}_{Z_n})=\mC$ and the statement follows.
If $w\not=w'$ then $C(w')$ has at least one minimal cup connecting say $i$ and $i+1$. Using the adjunctions \cite[Lemma 4.4]{CK} for cup and cap functors in we can remove this cup in expense of applying a cap functor $F_i[1]$ to $\LB_w$. Let $a$, $b$ be the $i$-th and $i+1$-st labels of $w$ and denote by $v$ the weight which is obtained from $w$ by removing these these two points.
Then by \cite[6.3]{CK} we have the following four cases: $F_i\LB_w=0$ if $ab=\down\down$ or  $ab=\up\up$, and then of course $\Ext^i_{\Coh(Z_n)}(\LB_w,\LB_w)=\{0\}$. We have $F_i\LB_w\cong \LB_{v}[1]$ if $ab=\up\down$, in which case the claim follows by induction (note that we removed a clockwise cup/cap). We have $F_i\LB_w\cong \LB_{v}$ if $ab=\down\up$, in which case the claim follows by induction noting that we removed a counter-clockwise cup/cap. Hence the statement is true for minimal $w$.
Assume $w$ is not minimal. Choose a minimal cup in $C(w)$ say at the vertices $i, i+1$. Applying again adjunction properties, we can remove this cup by the expense of a cap functor $F_i[-1]$. If this cap creates a circle, we have $\Ext^\bullet_{\Coh(Z_n)}(\LB_w,\LB_{w'})\cong\Ext^{\bullet}_{\Coh(Z_n)}(\LB_v,\LB_{v'})\oplus \Ext^{\bullet+2}_{\Coh(Z_n)}(\LB_v,LB_{v'})$ by \cite[Corollary 5.10]{CK}. Since $d(w,w')=d(v,v')$, the statement follows.
If this cap does not create a circle, and $\Ext^\bullet_{\Coh(Z_n)}(\LB_w,\LB_w')\not=\{0\}$, then using again \cite[Corollary 5.10]{CK} and adjointness properties we get
\begin{eqnarray*}
&&\Ext^{\bullet+1}_{\Coh(Z_n)}(\LB_w,\LB_w')\oplus\Ext^{\bullet-1}_{\Coh(Z_n)}(\LB_w,\LB_w')\\
&\cong&\Ext^\bullet_{\Coh(Z_n)}(G_iF_i\LB_w,\LB_w')\\
&\cong&\Ext^\bullet_{\Coh(Z_n)}(G_iF_i\LB_x,\LB_w')\\
&\cong&\Ext^\bullet_{\Coh(Z_n)}(\LB_x,F_iG_i\LB_w')\\
&\cong&\Ext^{\bullet+1}_{\Coh(Z_n)}(\LB_w,\LB_z)\oplus\Ext^{\bullet-1}_{\Coh(Z_n)}(\LB_w,\LB_z),
\end{eqnarray*}
where $z$ is obtained from $w'$, and $x$ is obtained from $w$, by swapping the labels at the vertices $i$ and $i+1$. In particular,
$$\Ext^{\bullet}_{\Coh(Z_n)}(\LB_x,\LB_w')\cong\Ext^{\bullet}_{\Coh(Z_n)}(\LB_w,\LB_z).$$ On the other hand $d(x,{w'})=d(w,z)$. (To see this assume first vertex $l$ and $k$ are connected to the vertices $i$ and $i+1$ via a cup diagram in $C(w')$).
\end{proof}

The following is now a direct consequence of \cite[Lemma 3]{Bezt}:

\begin{theorem}
There exists a unique $t$-structure of $D^b(\Coh(Z_n))$, such that the $\cM_w$'s form the simple objects.
\end{theorem}

\begin{proof}
We only have to verify the assumptions of \cite[Lemma 3]{Bezt}. These are however just Lemma~\ref{semisimple} together with the observation that the $\cM_w$'s are sheaves (so that $\operatorname{Hom}_{D^b(\Coh(Z_n))}(\cM_w, \cM_{w'}[l])=\{0\}$ for any positive $l$).
\end{proof}

Following Bezrukavnikov, we call this the {\bf exotic} $t$-structure.  We call the heart of this $t$-structure the category of exotic sheaves $\Ex_n$. The main result of this section is the following:

\begin{theorem}
\label{highestweight}
 There is a highest weight structure on $\Ex_n$ such the the sheaves $\LB_w$ are standard.
\end{theorem}

\begin{lemma}\label{comp-factors}
 The sheaf $\LB_w$ is exotic, and its composition factors are all of the form $\cM_{w'}$ for $w'\leq w$, with $\cM_{w}$ appearing exactly once.
\end{lemma}
\begin{proof}
 We induct on both the number of points, and the ordering given above.  Our base case is still that where $C(w)$ is empty, where this is obvious.

 As we noted before, we can write $w$ as $v^+$ for some sequence $v$
 on fewer points.  Recall the exact sequence
 (\ref{cup-exact}).  Now, by induction on the number of points
 $\vp_i(\LB_v)$ is exotic, and has the desired composition series
 (since $\cM_v$ appears once in $\LB_v$, we have
 $\cM_{w}=\vp_i(\cM_{v})$ appearing once), and by induction on the
 partial order, $\LB_{v^-}$ is exotic, and all its composition factors
 are strictly smaller than $w$.
\end{proof}

\begin{lemma}
\label{exsequence}
 The line bundles $\LB_v$ form an exceptional sequence, that is, we have $$\Ext^\bullet_{\Coh(Z_n)}(\LB_w,\LB_v)=0,$$ for all $v\ngeq w$.
\end{lemma}

\begin{proof}
 As usual, we have $\Ext^i(\LB_w,\LB_v)\cong H^i(\LB_w^*\otimes \LB_v)$.  Thus our problem reduces to computing the cohomology of certain line bundles.

Consider the map $\pi:Z_n\to Z_{n-1}$ given by forgetting the top space.  We note that if $\boldsymbol{\ep}=(\ep_1,\ldots,\ep_{n-1})$ is a vector valued in $\{1,0,-1\}$, then
\begin{equation*}
  \pi_*(\LB_{\boldsymbol{\ep}}\otimes\LB_n^{j})\cong
  \begin{cases}
   \LB_{\boldsymbol{\ep}}\otimes\Sym^{-j}(W) & j\leq 0\\
    0 & j=1\\
    \LB_{\boldsymbol{\ep}}\otimes\Sym^{j-2}(W) [-1]& j\geq 2
  \end{cases}
\end{equation*}
where $W\cong\pi_*\LB_n$ is a rank 2 vector bundle which is an extension
\begin{equation*}
0\to\LB_{n-1}^{-1}\to W\to \LB_{n-1}\to 0.
\end{equation*}
Thus, if a vector bundle is an extension of line bundles of the form $\LB_{\boldsymbol{\ep}}\otimes\LB_n^j[m]$, for $|j-1|\leq k$, then its push-forward is an extension of ones of the form $\LB_1^{\ep_1}\otimes\cdots\otimes \LB_{n-2}^{\ep_{n-2}}\otimes \LB_{n-1}^{j'}[m']$ where $|j'-1|\leq k+\ep_{n-1}$.

Applying this inductively, we see that the $\ell-n$-fold push-forward
$\pi_*^{\ell,n}\LB_w^*\otimes \LB_v$ is an extension of line bundles
of the form $\LB_{\boldsymbol{\ep}}\otimes\LB_n^j[m]$ where $|j-1|\leq
g_n+1$ where $g_n$ is the difference between the number of $\vee$'s in
the last $\ell-n$ places of $w$ and those in those places in $v$.  If
this number is ever negative, then $j=-1$, so the $\ell-n+1$-fold
push-forward is trivial.  Thus, if this push-forward is non-trivial, we must have this number always non-negative, that is, we must have $v\geq w$.
\end{proof}

\begin{proof}[Proof of Theorem~\ref{highestweight}]
Lemmata~\ref{comp-factors}~and~\ref{exsequence} show the line bundles $\LB_w$'s are standard covers of the simple modules $\cM_w$.  This shows that an object has negative $\Ext$ vanishing with all $\LB_w$ if and only if it does with $\cM_w$ (since $\Ext^i(\cM_w,X)=\Ext^i(\LB_w,X)$ for $i<0$ if $\Ext^i(\cM_v,X)=0$ for all $i<0$ and $v< w$), and the Serre subcategory generated by $\{\LB_v[i]\}_{i\geq 0}$ is the same as that generated by $\{\cM_{v}[i]\}_{i\geq 0}$.  That is, the exotic $t$-structure is exactly the one which Bezrukavnikov calls the $t$-structure of the exceptional sequence $\{\LB_v\}$. By \cite[Proposition 2]{Bezt}, the heart of this $t$-structure is highest weight, with $\{\LB_w\}$ as its standards.
\end{proof}

\providecommand{\bysame}{\leavevmode\hbox to3em{\hrulefill}\thinspace}
\providecommand{\MR}{\relax\ifhmode\unskip\space\fi MR }
\providecommand{\MRhref}[2]{%
  \href{http://www.ams.org/mathscinet-getitem?mr=#1}{#2}
}
\providecommand{\href}[2]{#2}


\begin{thebibliography}{BLPW08}

\bibitem[Ann]{Anno}
R.~Anno, \emph{Affine tangles and irreducible exotic sheaves},
  \textsf{arXiv:math/0802.1070}.

\bibitem[Bez03]{Bezt}
R.~Bezrukavnikov, \emph{Quasi-exceptional sets and equivariant coherent sheaves
  on the nilpotent cone}, Represent. Theory \textbf{7} (2003), 1--18.

\bibitem[Bez06]{BezNon}
\bysame, \emph{Noncommutative counterparts of the {S}pringer resolution},
  International Congress of Mathematicians. Vol. II, Eur. Math. Soc., Z\"urich,
  2006, pp.~1119--1144.

\bibitem[BGG76]{BGG}
I.~N. Bernstein, I.~M. Gelfand, and S.~I. Gelfand, \emph{A certain category of
  {${\mathfrak g}$}-modules}, Funkcional. Anal. i Prilozen. \textbf{10} (1976),
  no.~2, 1--8.

\bibitem[BGS96]{BGS}
A.~Beilinson, V.~Ginzburg, and W.~Soergel, \emph{Koszul duality patterns in
  representation theory}, J. Amer. Math. Soc. \textbf{9} (1996), no.~2,
  473--527.

\bibitem[BLPW08]{BLPW}
T.~Braden, A.~Licata, N.~Proudfoot, and B.~Webster, \emph{Gale duality and
  {K}oszul duality}, 2008, To appear in Advances in Mathematics.
  \textsf{arXiv:0806.3256}.

\bibitem[BS08a]{BS1}
J.~Brundan and C.~Stroppel, \emph{Highest weight categories arising from
  {K}hovanov's diagram algebra {I}: {C}ellularity}, 2008, arXiv:0806.1532.

\bibitem[BS08b]{BS3}
\bysame, \emph{Highest weight categories arising from {K}hovanov's diagram
  algebra {III}: Category $\mathcal{O}$}, 2008, arXiv:0812.1090.

\bibitem[BS10]{BS2}
\bysame, \emph{Highest weight categories arising from {K}hovanov's diagram
  algebra {II}: {K}oszulity}, Transf. groups (2010), no.~15, 1--45.

\bibitem[Car80]{RC}
A.~Rocha~{-} Caridi, \emph{Splitting criteria for {${\mathfrak g}$}-modules
  induced from a parabolic and the {B}ernstein-{G}elfand-{G}elfand resolution
  of a finite-dimensional, irreducible {${\mathfrak g}$}-module}, Trans. Amer.
  Math. Soc. \textbf{262} (1980), no.~2, 335--366.

\bibitem[CG97]{CG97}
N.~Chriss and V.~Ginzburg, \emph{Representation theory and complex geometry},
  Birkh\"auser, 1997.

\bibitem[CK06]{ChK}
Y.~Chen and M.~Khovanov, \emph{An invariant of tangle cobordisms via
  subquotients of arc rings}, 2006, \textsf{arXiv:QA/0610054}.

\bibitem[CK08]{CK}
S.~Cautis and J.~Kamnitzer, \emph{Knot homology via derived categories of
  coherent sheaves. {I}. {T}he {${\mathfrak{sl}}(2)$}-case}, Duke Math. J.
  \textbf{142} (2008), no.~3, 511--588.

\bibitem[CKS03]{CKS03}
A.~C{\u{a}}ld{\u{a}}raru, S.~Katz, and E.~Sharpe, \emph{D-branes, {$B$} fields,
  and {E}xt groups}, Adv. Theor. Math. Phys. \textbf{7} (2003), no.~3,
  381--404.

\bibitem[DK09]{DK}
C.~Dodd and K.~Kremnitzer, \emph{A localization theorem for finite
  {W}-algebras}, 2009, \textsf{arXiv:0911.2210}.

\bibitem[Ful97]{Fulton}
W.~Fulton, \emph{Young tableaux}, LMS Student Texts, vol.~35, Cambridge
  University Press, 1997.

\bibitem[Fun03]{Fung}
F.~Fung, \emph{On the topology of components of some {S}pringer fibers and
  their relation to {K}azhdan-{L}usztig theory}, Adv. Math. \textbf{178}
  (2003), no.~2, 244--276.

\bibitem[FW99]{FW99}
D.~S. Freed and E.~Witten, \emph{Anomalies in string theory with {D}-branes},
  Asian J. Math. \textbf{3} (1999), no.~4, 819--851.

\bibitem[GH78]{GH}
P.~Griffiths and J.~Harris, \emph{Principles of algebraic geometry}, John Wiley
  \& Sons, 1978, Pure and Applied Mathematics.

\bibitem[Gin08]{Gin08}
V.~Ginzburg, \emph{{H}arish-{C}handra bimodules for quantized {S}lodowy
  slices}, 2008, \textsf{arXiv:0807.0339}.

\bibitem[Har77]{Hartshorne}
R.~Hartshorne, \emph{Algebraic geometry}, Springer-Verlag, 1977, Graduate Texts
  in Mathematics, No. 52.

\bibitem[Hum08]{HumphreysO}
J.~E. Humphreys, \emph{Representations of semisimple {L}ie algebras in the
  {BGG} category {$\scr{O}$}}, Graduate Studies in Mathematics, vol.~94,
  American Mathematical Society, Providence, RI, 2008.

\bibitem[Huy05]{Huy}
D.~Huybrechts, \emph{Complex geometry}, Universitext, Springer-Verlag, 2005.

\bibitem[Kho00]{KhoJones}
M.~Khovanov, \emph{A categorification of the {J}ones polynomial}, Duke Math. J.
  \textbf{101} (2000), no.~3, 359--426.

\bibitem[Kho02]{Khotangles}
\bysame, \emph{A functor-valued invariant of tangles}, Algebr. Geom. Topol.
  \textbf{2} (2002), 665--741.

\bibitem[KS02]{KS02}
S.~Katz and E.~Sharpe, \emph{D-branes, open string vertex operators, and {E}xt
  groups}, Adv. Theor. Math. Phys. \textbf{6} (2002), no.~6, 979--1030 (2003).

\bibitem[Leu02]{Leu02}
N.~C. Leung, \emph{Lagrangian submanifolds in hyper{K}\"ahler manifolds,
  {L}egendre transformation}, J. Differential Geom. \textbf{61} (2002), no.~1,
  107--145.

\bibitem[Man07]{Manolescu}
C.~Manolescu, \emph{Link homology theories from symplectic geometry}, Adv.
  Math. \textbf{211} (2007), no.~1, 363--416.

\bibitem[MP06]{Melnikov}
A.~Melnikov and N.~G.~J. Pagnon, \emph{On intersections of orbital varieties
  and components of {S}pringer fiber}, J. Algebra \textbf{298} (2006), no.~1,
  1--14.

\bibitem[MV07]{MV08}
I.~Mirkovi\'{c} and M.~Vybornov, \emph{{Q}uiver varieties and
  {B}eilinson-{D}rinfeld {G}rassmannians of type {A}}, 2007,
  \textsf{arXiv:0712.4160}.

\bibitem[RT90]{RT}
N.~Yu. Reshetikhin and V.~G. Turaev, \emph{Ribbon graphs and their invariants
  derived from quantum groups}, Comm. Math. Phys. \textbf{127} (1990), no.~1,
  1--26.

\bibitem[Sha04]{Sha04}
E.~Sharpe, \emph{Mathematical aspects of {D}-branes}, Quantum theory and
  symmetries, World Sci. Publ., Hackensack, NJ, 2004, pp.~614--620.

\bibitem[Spa76]{Spa76}
N.~Spaltenstein, \emph{The fixed point set of a unipotent transformation on the
  flag manifold}, Nederl. Akad. Wetensch. Proc. Ser. A {\bf 79} (Indag. Math.)
  \textbf{38} (1976), no.~5, 452--456.

\bibitem[SS06]{SS}
P.~Seidel and I.~Smith, \emph{A link invariant from the symplectic geometry of
  nilpotent slices}, Duke Math. J. \textbf{134} (2006), no.~3, 453--514.

\bibitem[Str05]{StDuke}
C.~Stroppel, \emph{Categorification of the {T}emperley-{L}ieb category,
  tangles, and cobordisms via projective functors}, Duke Math. J. \textbf{126}
  (2005), no.~3, 547--596.

\bibitem[Str09]{StrSpringer}
\bysame, \emph{Parabolic category {$\scr O$}, perverse sheaves on
  {G}rassmannians, {S}pringer fibres and {K}hovanov homology}, Compos. Math.
  \textbf{145} (2009), no.~4, 954--992.

\bibitem[Var79]{Var79}
J.~A. Vargas, \emph{Fixed points under the action of unipotent elements of
  {${\rm SL}\sb{n}$} in the flag variety}, Bol. Soc. Mat. Mexicana (2)
  \textbf{24} (1979), no.~1, 1--14.

\bibitem[Web]{WebWO}
B.~Webster, \emph{Singular blocks of parabolic category $\mathcal{O}$ and
  finite {W}-algebras}, \textsf{arXiv:0909.1860}.

\end{thebibliography}
\end{document}